\documentclass{amsart}
\usepackage{graphicx}
\usepackage{hyperref}
\usepackage{pslatex}
\usepackage{amsmath,amsfonts}
\usepackage{pgfplots}
\pgfplotsset{compat=1.15}
\usepackage{mathrsfs}
\usepackage{rotating}
\usepackage{soul}
\usepackage{amssymb}
\usepackage{tikz}
\usepackage{verbatim}
\usetikzlibrary{arrows}
\usepackage{color}
\usepackage[utf8]{inputenc}
\newtheorem{theorem}{Theorem}[section]
\newtheorem{lemma}[theorem]{Lemma}

\newtheorem*{theorem*}{Theorem}

\newtheorem{corollary}[theorem]{Corollary}

\theoremstyle{definition}

\newtheorem{conjecture}[theorem]{Conjecture}

\theoremstyle{remark}
\newtheorem{remark}[theorem]{Remark}

\numberwithin{equation}{section}
\definecolor{qqwuqq}{rgb}{0,0.39215686274509803,0}

\def\R{\mathbb R}
\def\C{\mathbb C}
\def\Z{\mathbb Z}
\def\N{\mathbb N}
\def\f{\widehat{f}}

\def\supp{\text{supp}}
\def\({\left(}
\def\){\right)}
\def\[{\left[}
\def\]{\right]}
\def\<{\left<}
\def\>{\right>}

\def\eps{\epsilon}

\begin{document}
	\title[Aleksei Kulikov, Lucas Oliveira, and Jo\~ao P. G. Ramos]{On Gaussian decay rates of harmonic oscillators and equivalences of related Fourier uncertainty principles}
	\author{Aleksei Kulikov}
	\address{Department of Mathematical Sciences, Norwegian University of Science and Technology, NO-7491 Trondheim, Norway}
	\email{lyosha.kulikov@mail.ru}
	\author{Lucas Oliveira}
	\address{Departamento de Matem\'{a}tica\\UFRGS\\ Porto Alegre RS 91509-900}
	\email{lucas.oliveira@ufrgs.br}
	\author{Jo\~{a}o P. G. Ramos}
	\address{ETH Z\"urich D-MATH \\ R\"amistrasse 101, 8092 Z\"urich, Switzerland}
	\email{joao.ramos@math.ethz.ch}

	\maketitle
	
	\begin{abstract}
		We make progress on a question by Vemuri on the optimal Gaussian decay of harmonic oscillators, proving the original conjecture up to an arithmetic progression of times. The techniques used are a suitable translation of the problem at hand in terms of the free Schr\"odinger equation, the machinery developed in the work of Cowling, Escauriaza, Kenig, Ponce and Vega \cite{Escauriaza2010}, and a lemma which relates  decay on average to pointwise decay. 
		
		Such a lemma produces many more consequences in terms of equivalences of uncertainty principles. Complementing such results, we provide endpoint results in particular classes induced by certain Laplace transforms, both to the decay Lemma and to the remaining cases of Vemuri's conjecture, shedding light on the full endpoint question. 
	\end{abstract}
	
	\section{Introduction}
	\subsection{Historical Background}
	Uncertainty principles have permeated Mathematics and Physics for many years, since the introduction of such a concept by Heisenberg in the context of Quantum Mechanics. For the Fourier transform 
	\begin{equation}\label{fourier_transform}
		\widehat{f}(\xi)=\int_{\R}e^{-2\pi i x\cdot\xi}f(x)\,dx\,,
	\end{equation}
	\emph{Heisenberg's uncertainty principle} can be stated simply as
	\begin{equation*}\label{heisenberg}
		\int_{\R}|f(x)|^2dx\le 4\pi\left(\int_{\R}|x|^2|f(x)|^{2}dx\right)^{1/2}\left(\int_{\R}|\xi|^2|\widehat{f}(\xi)|^{2}d\xi\right)^{1/2}.    
	\end{equation*}
	This inequality states essentially that we cannot concentrate in space and frequency sides too much simultaneously, and has the physical interpretation that we cannot make measurements about the position and momentum of a particle (in the probabilistic sense) with high precision for both. 
	
	Further than Heisenberg's initial contribution, there are many other instances and kinds of uncertainty principles. Benedicks's uncertainty principle \cite{Benedicks}, for example, predicts that for $f\in L^{1}(\R^{d})$  the measures of the sets $\{x\in\R^{d}:f(x)\neq0\}$ and $\{x\in\R^{d}:\f(x)\neq0\}$ cannot be both finite, unless $f\equiv0$. The Amrein--Berthier  Uncertainty Principle \cite{Amrein-Berthier} complements the previous one stating that for any $f\in L^{2}(\R^{d})$ and any pair of finite measure sets $E,F\subseteq \R^{d}$ there is a positive constant $C=C(E,F)$ such that
	$$
	\|f\|_{L^{2}}^{2}=C\left(\int_{E^{c}}|f(x)|^{2}dx+\int_{F^{c}}|\f(x)|^{2}dx\right).
	$$

	Recently, a different kind of uncertainty principle related to sign changes of the Fourier transform has attracted some attention. For example, in \cite{Bourgain-Clozel-Kahane} Bourgain, Clozel and Kahane had proved that $f$ and $\widehat{f}$ cannot simultaneously concentrate negative mass on arbitrarily small neighbourhoods of the origin. For further developments in this direction, see \cite{Cohn-Goncalves, Goncalves-Oliveira-Ramos1, Goncalves-Oliveira-Ramos2, Goncalves-Oliveira-Steinerberger} and references therein. 
	
	In this work we are concerned with uncertainty principles that are, to some extent, related to the properties of the Gaussians. The first such result was obtained by Hardy \cite{Hardy} in 1933 and can be stated in the following way: if $|f(x)|\le A^{-\pi x^2}$ and $|\f(x)|\le Be^{-\pi x^{2}}$, then there is a constant  $C$ such that $f(x)=\f(x)=Ce^{-\pi x^{2}}$. In fact, Hardy proved more than this:
	\begin{itemize}
		\item If $f$ and $\f$ are of order $O(|x|^{m}e^{-\pi x^{2}})$ for some $m$ and for large $x$, then $f$ is a linear combination of Hermite functions;
		\item If $f$ is $O(e^{-\pi x^{2}})$ and $\f$ is $o(e^{-\pi x^{2}})$ (or vice-versa), then $f = 0$. 
	\end{itemize}
	On the other hand, it is not enough to assume that $|f(x)|\le A e^{-\lambda\pi x^{2}}$ and $|\widehat{f}(\xi)|\le A e^{-\mu\pi \xi^{2}}$ for some $0<\lambda,\mu<1$ since 
	nontrivial functions $f$ satisfying these conditions form an infinite dimensional space.
	
	The techniques coming from Complex Analysis (to be precise, the Phrmagm\'en-Lindel\"of principle) were decisive in the proof of the above result, as well as in the proof of the following extension of it obtained by Beurling in 1964 (whose proof seems to have been lost, until Hörmander \cite{Hormander} in 1991 provided a full proof, based on personal notes taken during a discussion of this result with Beurling himself): if $f\in L^{1}(\R)$ is such that
	\begin{equation*}\label{beurling}
		B(f,\f):=\int_{\R}\int_{\R}|f(x)||\f(y)|e^{2\pi|xy|}dx\,dy<\infty\,,
	\end{equation*}
	then $f\equiv0$.
	
	It is worth mentioning that interesting  generalisations of this result, in an almost subcritical level, have been  recently obtained by Bonami-Demange \cite{Bonami-Demange}, Hedenmalm \cite{Hedenmalm} and by Gao \cite{Gao}; see also \cite{Hormander}. 
	
	In a different direction, and also relevant to our current work, are the generalisations of Hardy's uncertainty principle where we can combine different kinds of Gaussians and different control of $L^{p}$-norms, where Hardy's theorem can be seen as an $L^{\infty}$-norm version of a more general principle. Major contributions along these lines were obtained by Cowling and Price \cite{Cowling-Price} and Morgan \cite{Morgan}.
	
	In our context, Cowling--Price's Uncertainty Principle can be stated in the following way: $\|fe^{ax^2}\|_{L^{p}}<\infty$ and $\|f e^{bx^2}\|_{L^{q}}<\infty$ imply $f\equiv 0$ when $ab>\pi^2$. When $ab<\pi^2,$ in the same way as in Hardy's theorem, there are nontrivial examples of functions $f$ satisfying these conditions.
	
	More recently, Hardy's Uncertainty Principle has been shown to be related to the study of decay behaviour of evolution equations. Indeed, let us consider the question of uniqueness for solutions of Schr\"odinger evolutions of the kind
	\begin{equation}\label{NLS}
		i\frac{\partial u}{\partial t}+\frac{\partial^{2}u}{\partial x^{2}}+F(u,\overline{u})=0.
	\end{equation}
	In other terms, we are interested in determining when two solutions $u_1,u_2$ of \eqref{NLS} coincide, given they are equal on a set $S \subset (0,+\infty) \times \R.$ Escauriaza \emph{et al} \cite{Escauriaza2006} extended such a study by observing that Hardy's result may be reformulated as the property that, if a solution to
	\begin{equation}\label{LS}
		i\frac{\partial u}{\partial t}+\frac{\partial^{2}u}{\partial x^{2}}=0
	\end{equation}
	has sufficient Gaussian decay at two different times, it must vanish identically. In line with this, exploring techniques and ideas based on convexity of solutions of Schr\"odinger equations such as \eqref{NLS} with additional Gaussian control, in \cite{Escauriaza2008}, Euscariaza \emph{et al} observed that such solutions should satisfy a weak version of Hardy's Uncertainty Principle. In \cite{Escauriaza2010}, Cowling \emph{et al} showed a real variable proof of Hardy's and Cowling-Price's Uncertainty Principles. Their result may be summarised as follows: if $v(x,t)$ is a solution of the free Schr\"odinger equation \eqref{LS}, with initial condition $v(x,0)=O(e^{- \alpha x^{2}}),$ and such that $v(x,T)=O(e^{- \beta x^{2}})$ for some $T>0$ with $T\alpha\beta>\pi^{2}$, then $v\equiv0$.
	\subsection{Main results}
	We were able to show that $L^p$-bounds on a function and its Fourier transform imply pointwise bounds up to an $\eps$ in the exponent. For example, we prove that if 
	$$
	\int_{\R^n}|f(x)|^{2}e^{2\pi \alpha  |x|^{2}}\,dx<\infty \mbox{ and } \int_{\R^n}|\widehat{f}(x)|^{2}e^{2\pi \alpha  |x|^{2}}\,dx<\infty
	$$
	then, for each $\epsilon>0$ there is a positive constant $A=A(\epsilon)$ such that
	$$
	|f(x)|\le A_{\epsilon}e^{-(1-\epsilon)a\pi |x|^{2}},
	$$
	in particular this implies that Cowling--Price's Uncertainty Principle follows from the Hardy's one.

	This result is inspired by an attempt to attack a conjecture of Vemuri \cite{Vemuri} about the decay of solutions of the quantum harmonic oscillator. For $f, g:\R^n\to \C$ we denote $C^p(f, g) = ||fg||_{L^p}^p$ and for $g(x) = e^{2\pi a|x|^2}$ we write $C^p(f, g) = C^p_a(f)$. Consider the class of functions
	\begin{equation}\label{Easpaces}
		E_{a}^{p}(n)=\{f:\R^n\to \C:\mid C^p_a(f) < \infty \mbox{ and } C^p_a(\hat{f}) < \infty\}.
	\end{equation}
	In terms of these spaces, the result we formulated in the beginning of this section can be restated in the following form: if
	$
	f\in E_{a}^{2}(n)$ then for all $\epsilon>0$ we have $f\in E_{a-\epsilon}^{\infty}(n)$. 
	
	Let $H := - \Delta + 4\pi^2 |x|^2$ denote the (normalised) Quantum Harmonic Oscillator. Fix $f \in L^2(\R^n).$ We define $\Phi(x,t)$ to be the solution of the time-dependent initial value problem 
	\begin{equation}\label{eq:quantum-harmonic}
		\begin{cases}
			i \partial_t \Phi = H \Phi, \, \, & \text{for} \, (x,t) \in \R^n \times \R; \cr 
			\Phi(x,0) = f(x), \,\, & \text{on} \,\R^n. 
		\end{cases}
	\end{equation}
	The solution $\Phi$ to this problem is intimately related to the Hermite functions when $n=1$. Indeed, if we have 
	$$
	f(x) = \sum_{k \ge 0} a_k h_k(x),
	$$
	then we may write the solution above at time $t \in \R$ as 
	$$
	\Phi(x,t) = \sum_{k \ge 0} e^{2\cdot (2k+1)\pi i t} a_k h_k(x),
	$$
	where we define our normalisation of the Hermite functions $\{h_k\}_{k \ge 0}$ to be the complete orthonormal system in $L^2$ such that $\mathcal{F}(h_k) = (-i)^k h_k.$ This formula for the solution converges in a pointwise sense for $f$ in the Schwartz space $\mathcal{S}(\R)$. From now on, we will use the notation $\Phi f(x,t)$ to denote the solution to \eqref{eq:quantum-harmonic} with initial value $f$. Whenever it is obvious from context, we shall simply write $\Phi(x,t)$ as above. 
	
	With these definitions, Vemuri's conjecture \cite{Vemuri} states that, if $f \in E^{\infty}_{\tanh(2\alpha)}(n),$ then 
	$$
	\Phi(\cdot, t) \in E^{\infty}_{\tanh(\alpha)}(n), \, \forall t >0.
	$$
	In fact, Vemuri proved that $\Phi(\cdot,t) \in E^{\infty}_{\tanh(\alpha)-\eps}(n), \, \forall \eps > 0.$ 
	
	By relating the evolution of the Harmonic Oscillator problem to the Schr\"odinger equation and the optimal decay for Schr\"odinger evolutions as in, for instance, \cite{Escauriaza2010}, we obtain an $L^2$-version of Vemuri's conjecture: if $C_a^2(f),C_a^2(\widehat{f}) < + \infty,$ then 
	$$C_{\tanh(\alpha)}^2(\Phi(\cdot,t))<+\infty, $$
	where $a = \tanh(2\alpha).$ Our first main result is, as far as we know, the first step towards settling Vemuri's conjecture in the original $L^{\infty}$ case.
	
	\begin{theorem}\label{thm:vemuri} Let $f \in E^{\infty}_{\tanh(2\alpha)}(n),$ for some $\alpha >0.$ Then $\Phi(\cdot,t) \in E^{\infty}_{\tanh(\alpha)}(n)$ whenever $t \not\in \{\frac{1}{16} + \frac{k}{8}, \, k \in \Z\}.$ 
	\end{theorem}
	
	We will, in fact, prove that Vemuri's conjecture can be \emph{sharpenned} in the case $t \not\in \{\frac{1}{16} + \frac{k}{8}, \, k \in \Z\}.$  That is, the largest $b > 0$ for which $\Phi(\cdot,t) \in E^{\infty}_b$ satisfies $b > \tanh(\alpha),$ whenever $t$ is not in the exceptional set above. 
	
	The techniques used in order to prove Theorem \ref{thm:vemuri} are based on several recent results in the literature involving Gaussian decay of Schr\"odinger equations. Indeed, we first make use of a change of variables which takes the evolution of the harmonic oscillator into that of the free Schr\"odinger equation. Although we provide an alternative proof of such Lemma, we note that this kind of formulas seems to be known in the physics literature; see, for instance, \cite{Takagi1, Takagi2}. It was pointed to us recently that such changes of variables have also been employed in a similar context by B. Cassano and L. Fanelli in \cite{CassanoFanelli2016} (see also \cite{CassanoFanelli2013, BarceloBiagioFanelli2021} and the references therein). 
	
	We use a change of variables which preserves the free Schr\"odinger equation, in the same spirit as in \cite{Escauriaza2010}, in order to be able to use the original results by Escauriaza, Kenig, Ponce and Vega on convexity properties of Gaussian decay of Schr\"odinger equations. Finally, the last technique used is the mechanism described above to pass from $L^2$ to $L^{\infty},$ and vice versa. 
	
	It is worth to mention, though, that, in order to achieve such a result in higher dimensions, we will need a version of the Gaussian observation above for all dimensions. This is achieved through the following result: 
	
	\begin{lemma}\label{lemma:decay}
		Suppose that $w:\R^{d}\to [1,\infty)$ is a measurable function and $f:\R^{d}\to \C$ is a $C^{\infty}$ function which are related by the following assumptions:
		\begin{itemize}
			\item [(i)] For some $1\le p<\infty$ we have
			\begin{equation}
				\int_{\R^n}|f(x)|^{p}w(x)^{p}\,dx<\infty\,; 
			\end{equation}
			\item [(ii)] The sets $\{x:w(x)<t\}$ are convex for each $t>1$;
			\item [(iii)] There is $1\le r\le \infty$ such that for all $m\in \N_0$  we have $\nabla^m f\in L^{r}(\R^{d})$.
		\end{itemize}
		Then, for each $\epsilon>0$ and each $m\in \N_0$, there is a constant $A_{m, f,\epsilon}$ such that
		\begin{equation}
			|\nabla^m f(x)|\le A_{m,f, \epsilon} w(x)^{-(1-\epsilon)} \quad \forall x\in\R^{d}\,.
		\end{equation}
	\end{lemma}
	
	It has recently come to our attention that a version of such result is known in dimension 1 from \cite[Theorem~1.7]{KwongZettl}. As we could not find a suitable reference for the higher-dimensional result, we decided to include it here together with its proof, as it is also of independent interest. 
	With such a tool at hand, we get a sharp relation (up to the endpoint) between Hardy's, Cowling--Price's and Morgan's Uncertainty Principles in the sub-critical regime. As we've already recalled Hardy's and Cowling--Price's Uncertainty Principles above, we briefly recall Morgan's Uncertainty Principle below (in a generalized version obtained by Ben Farah and Mokni \cite{Farah-Mokni}). For that, we shall use the notation $e_{a,b}(x) = e^{a\pi |x|^b}.$ 
	
	\begin{theorem*}[Morgan; Ben Farah--Mokni]
		Suppose that $e_{a,\alpha}f\in L^{p}(\R^{d})$ and that $e_{b,\beta}\widehat{f}\in L^{q}(\R^{d})$ for $1\le p,q\le \infty$, $\alpha>2$ and $\beta=\alpha/(\alpha-1)$. Then we have the following conclusions
		\begin{itemize}
			\item If $(a\alpha)^{1/\alpha}(b\beta)^{1/\beta}>\sin^{1/\beta}\left(\frac{\pi}{2}(\beta-1)\right)$, then $f\equiv0$;
			\item If $(a\alpha)^{1/\alpha}(b\beta)^{1/\beta}<\sin^{1/\beta}\left(\frac{\pi}{2}(\beta-1)\right)$, then there  are nontrivial functions verifying both conditions.
		\end{itemize}
	\end{theorem*}
	Observe that, in all the situations mentioned, when we are in the \emph{subcritical situation}, the theorems do not provide a clear information about the behavior of the functions. Our goal is to provide better information about the structure and behavior of such functions, and additionally, to reformulate this as a kind of quantitative relation. In that regard, we have the following:

	\begin{corollary}[Subcritical estimates]\label{thm:main} 
		If the function $f:\R^d\to \C$ is such that $e_{a, \alpha}f\in L^p$ and $e_{a, \beta}\hat{f}\in L^q$ for some $a, b, \alpha, \beta > 0$ and $p, q\ge 1$ then for all $\eps > 0$ there exists $C = C(\eps, f)$ such that
		$|f(x)|\le Me^{-(1-\epsilon)a\pi |x|^{\alpha}}$.
	\end{corollary}

	

	Corollary \ref{thm:main} may be then seen as a step in order to convert $L^2$ results for Gaussian weights into $L^{\infty}$ ones. Indeed, in the supercritical case $ab \ge 1$ in Cowling--Price's UP, this result shows that the only relevant case is indeed $ab=1,$ as all others imply the hypotheses in Hardy's uncertainty principle. \\
	
	
	The last results which we prove in this paper address the question of the endpoint in both Theorem \ref{thm:vemuri} and Corollary \ref{thm:main}. Indeed, Theorem \ref{thm:vemuri} leaves, perhaps suggestively, the sequence $\{ (2k+1)/16\}_{ k \in \Z}$ out of its statement -- which contains the (dilated) version of eigenvalues of the harmonic oscillator. Furthermore, Corollary \ref{thm:main} leaves open the question of determining whether a function $f$ such that 
	$$
	\int_{\R}|f(x)|^{2}e^{2\pi \alpha  |x|^{2}}\,dx<\infty \mbox{ and } \int_{\R}|\widehat{f}(x)|^{2}e^{2\pi \alpha  |x|^{2}}\,dx<\infty
	$$
	automatically satisfies $f(x) e^{a\pi x^2} \in L^{\infty}.$ 
	
	In this direction, given a finite measure $\mu$ with support on the positive real line, we consider its Laplace transform
	\begin{equation}
		\mathcal{L}\mu(s) = \int_0^{+\infty} e^{-s t} \, d \mu(t)
	\end{equation}
	and let $\varphi(x) = \mathcal{L}\mu(\pi |x|^2).$ 
	
	\begin{theorem}\label{thm:Laplace-sharp-1} 
		If $\varphi \in E^2_{a},$ then $\varphi \in E^{\infty}_{a}.$ 
	\end{theorem}  
	
	For the question on the endpoint of Theorem \ref{thm:vemuri}, we consider a slightly different class of functions: indeed, as we shall see in subsection \ref{ssec:laplace-transforms-1}, the endpoint version of Corollary \ref{thm:main} is much easier for Laplace transforms of measures supported on the positive real line. 
	
	Nevertheless, one may still wonder whether this example may be suitably tweaked in order to obtain a class of functions for which Vemuri's conjecture is indeed \emph{sharp}. In fact, Vemuri himself obtained that, if 
	$$\mathcal{G}_a(x) := e^{- \pi (a + i \sqrt{1-a^2}) |x|^2},$$
	with $a = \tanh(2\alpha) \in (0,1),$ then $|\Phi \mathcal{G}_a(y,-1/16)| = C e^{-\pi \tanh(\alpha)|y|^2}.$ Inspired by this observation, we prove that the full version of Vemuri's conjecture, as well as the endpoint version of our main result, hold and are \emph{sharp} for a class of transforms based on the functions $\mathcal{G}_a$ above. 
	
	\begin{theorem}\label{thm:second-laplace-full} Let  $a = \tanh(2\alpha) \in (0,1),$ and 
		$$
		\varphi(x) = \int_0^1 \mathcal{G}_r(x) \, d\mu(r)
		$$
		for some finite measure $\mu.$ Then:
		\begin{enumerate}
			\item If $\varphi \in E^2_a,$ then $\varphi \in E^{\infty}_a;$
			\item If $\varphi \in E^{\infty}_a,$ then for all $\beta \in \R,$ we have $\Phi \varphi(\cdot,\beta) \in E^{\infty}_{\tanh(\alpha)}.$ 
		\end{enumerate} 
	\end{theorem}

	The structure of the article is as follows:
	\begin{itemize}
		\item In \textbf{Section 2}, we will prove Lemma \ref{lemma:decay}, as well as Lemma \ref{lemma:kolmogorov} which shows that strong enough $L^p$ bounds on $f$ and $\hat{f}$ imply that $f\in \mathcal{S}(\R^n)$.
		
		
		\item In \textbf{Section 3}, we will prove  Corollary \ref{thm:main} and Theorem \ref{thm:vemuri}. 
		
		\item Finally, in \textbf{Section 4}, we will prove Theorems \ref{thm:Laplace-sharp-1} and \ref{thm:second-laplace-full}, which introduce a large class of examples that verify the conclusion of Theorem \ref{thm:vemuri} and Corollary \ref{thm:main} \emph{without} the $\epsilon$ loss for the case of Gaussian type weights. In this part, besides our main results and techniques, we shall resort to complex analysis methods as well. 
	\end{itemize}

	\section{Main Lemmas}\label{sec:lemmas}
	The proof of Lemma \ref{lemma:decay} is based on the following higher-dimensional version of the Kolmogorov--Landau inequality. For the reader's convenience we provide a short proof of it.
	\begin{lemma}\label{lemma:kolmogorov}
		Let $\Omega = \{x = (x_1, \ldots, x_n)\in \R^n:\, x_1 > 0\}$ and let $f\in L^p(\Omega)\cap C^m(\bar{\Omega})$ for some $1 \le p \le \infty$, $m > n$. For each $0 \le k \le m - n$ there exists $C= C_{k, m, n}$ such that for all $1 \le r \le \infty$ we have
		\begin{equation}\label{Landau}
			|\nabla^k f(0)|\le C||f||_{L^p(\Omega)}^{\alpha}||\nabla^m f||_{L^r(\Omega)}^{1-\alpha}, 
		\end{equation}
		where $\alpha = \alpha(k, n, m, p, r) = \frac{m - k - \frac{n}{r}}{m + \frac{n}{p}-\frac{n}{r}}$.
	\end{lemma}
	Note that by applying an orthogonal transformation this lemma can be applied to any half-space in place of $\Omega$ and any point $p$ on its boundary in place of $0$.
	\begin{proof}
		First, we show that the right-hand side of \eqref{Landau} is positive unless $f$ is identically zero. Indeed, if $||f||_{L^p} = 0$ then $f$ is zero almost everywhere, hence zero identically since $f\in C^m(\bar{\Omega})$. Similarly, if $||\nabla^m f||_{L^r} = 0$ then $\nabla^m f$ is identically zero, hence $f$ is a polynomial of degree at most $m - 1$. But then it is not in $L^p(\Omega)$ unless it is identically zero.
		
		So, we can assume that both $||f||_{L^p}$ and $||\nabla^m f||_{L^r}$ are strictly positive. If one of them is infinite then there is nothing to prove. Let us consider the function $g(x) = af(bx)$ for some $a, b > 0$. Observe that estimates \eqref{Landau} for $f$ and for $g$ are equivalent due to our choice of $\alpha$ since both sides are multiplied by the same amount. By choosing appropriate numbers $a$ and $b$ we can without loss of generality assume that $||g||_{L^p} = ||\nabla^m g||_{L^r} = 1$ and we have to show that $|\nabla^k g(0)|\le C$.
		
		Let $U = [0, 1]\times [-\frac{1}{2}, \frac{1}{2}]^{n-1}$. Since $p, r\ge 1$ and the measure of $U$ is $1$ we have $||g||_{L^1(U)}\le 1$ and $||\nabla^m g||_{L^1(U)}\le 1$ by H\"{o}lder's inequality. It remains to use two well-known facts from the theory of Sobolev spaces: 
		\begin{enumerate}
			\item[(i)] the space $W^m_1(U)$ -- functions having weak derivatives up to order $m$ in $L^1(U)$ on the bounded Lipschitz domain $U$ -- continuously embeds into $C^{n-m}(\bar{U})$;
			\item[(ii)] for such spaces, considering only the norm of the function and its $m$-th derivative yields an equivalent norm.
		\end{enumerate} 
		Since $0\in \bar{U},$ we get the desired result.
	\end{proof}
	
	\begin{proof}[Proof of Lemma \ref{lemma:decay}]
		Let us fix $x_0\in \R^d$. Since the set $V = \{x\in \R^d\mid w(x) < w(x_0)\}$ is convex and $x_0\notin V$ we can find a half-space $\Omega$ such that $x_0$ is on its boundary and for all $x\in \Omega$ we have $w(x) \ge w(x_0)$. We have 
		$$\int_{\Omega} |f(x)|^pdx \le w(x_0)^{-p}\int_{\Omega} |f(x)|^p w(x)^pdx \le w(x_0)^{-p}\int_{\R^n} |f(x)|^pw(x)^pdx = w(x_0)^{-p}C_f$$
		and
		$$||\nabla^m f||_{L^r(\Omega)} \le ||\nabla^m f||_{L^r(\R^n)} = C_{f, m}.$$
		Applying Lemma \ref{lemma:kolmogorov} to $f$ we get for $m > n + k$
		$$|\nabla^k f(x_0)| \le C_{k, n, m}w(x_0)^{-\alpha}C_f^\alpha C_{f, m}^{1-\alpha}.$$
		Observe that for fixed $k, n, p, r$ given $\eps > 0$ for big enough $m$ we have $\alpha > 1 -\eps$. Choosing such an $m$ gives us the desired estimate.
	\end{proof}
	
	To verify condition (iii) of Lemma \ref{lemma:decay} we will use the following lemma which says that if $f$ and $\hat{f}$ decay faster than any polynomial on average then $f\in \mathcal{S}(\R^n)$.
	\begin{lemma}\label{lemma:schwartz}
		Let $f:\R^n\to \mathbb{C}$ and let $1 \le p, q < \infty$. If for all $m\in \N_0$ we have $$\int_{\R^d}|f(x)|^p(1+|x|)^{pm} < \infty$$ and $$\int_{\R^d}|\hat{f}(x)|^q(1+|x|)^{qm} < \infty$$ then $f\in\mathcal{S}(\R^n)$.
	\end{lemma}
	\begin{proof}
		First, we show that $\nabla^m f$ is bounded and continuous for all $m\in \N_0$. For a multi-index $\beta\in \N_0^n$ we have $\widehat{\partial^\beta f}(x) = (2\pi i)^{|\beta|}x^\beta \hat{f}(x)$. Thus, if $x^\beta \hat{f}(x)\in L^1(\R^n)$ then $\partial^\beta f$ is bounded and continuous. We have
		$$||x^\beta \hat{f}(x)||_{L^1(\R^d)} \le ||(1+|x|)^{|\beta|} \hat{f}(x)||_{L^1(\R^d)}\le ||(1+|x|)^m \hat{f}(x)||_{L^q(\R^d)} ||(1+|x|)^{|\beta|-m}||_{L^{\frac{q}{q-1}}(\R^d)}.$$
		If $m$ is chosen bigger than $|\beta| + n$ then this quantity is finite and thus $x^\beta \hat{f}(x)\in L^1(\R^n)$.
		
		In particular, we get that $f\in C^\infty(\R^d)$. To finish the proof of the lemma, we are going to apply Lemma \ref{lemma:decay} to $f$. Consider the weight $w(x) = (1+|x|)^m$. Functions $f$ and $w$ satisfy the assumptions of Lemma \ref{lemma:decay} with $p = p$, $r = \infty$. Thus, for all $\eps > 0$, in particular $\eps = \frac{1}{2}$, we have
		$$|\nabla^k f(x)|\le C(1+|x|)^{(\eps - 1)m} = C(1+|x|)^{-m/2}.$$
		Since $m$ and $k$ are arbitrary, we get that $f\in \mathcal{S}(\R^n)$.
	\end{proof}
	\section{Proof of Corollary \ref{thm:main} and Theorem \ref{thm:vemuri}} 
	\begin{proof}[Proof of Corllary \ref{thm:main}] With the tools we have at our disposal, Corollary \ref{thm:main} becomes a trivial consequence. Indeed, when we are treating the situation in Cowling-Price's uncertainty principle, since the estimates $|f(x)|e^{\alpha\pi|x|^{2}} \in L^p$ and $|\widehat{f}(x)|e^{\beta\pi|x|^{2}} \in L^q$ imply, by Lemma \ref{lemma:schwartz}, that $f\in \mathcal{S}(\R^{n})$, we are in position to apply Lemma \ref{lemma:decay} and obtain for each $\epsilon>0$ the existence of a constant $A_{\epsilon}>0$ such that
		$$
		|f(x)|\le A_{\epsilon}e^{-\pi \alpha(1-\epsilon)|x|^{2}},
		$$
		and the analogous estimate holds for the Fourier transform. The case of the Morgan uncertainty principle is entirely analogous, and thus we are done. 
	\end{proof} 
	
	We now move on to the proof of our main theorem. 	
	
	\begin{proof}[Proof of Theorem \ref{thm:vemuri}] The proof will be divided into several steps. \\
		
		\noindent \textit{Step 1. Translating between the Quantum Harmonic Oscillator and the linear Schr\"odinger equation.} As we saw in the introduction, there is a simple way to write solutions to  write solutions to \eqref{eq:quantum-harmonic} in terms of the Hermite basis. We will use this connection, and the action of the Schr\"odinger evolution, to provide a simple proof of the link between \eqref{eq:quantum-harmonic} and the Schr\"odinger equation. 
		
		Before continuing, we introduce some notation for Hermite eigenfunctions of the Fourier transform in higher dimensions. For a multi-index $(\alpha_1,\alpha_2,\dots,\alpha_n) = \alpha \in \N_0^n$, we define the \emph{Hermite function of order $\alpha$} as 
		$$
		\mathbf{h}_{\alpha}(x) = \Pi_{i = 1}^n h_{\alpha_i}(x_i).
		$$
		We know from \cite[Lemma~11]{Goncalves} that 
		\begin{align*}
			& e^{it\Delta} (\mathbf{h}_{\alpha})(x) = \cr  
			& (1+4\pi i t)^{-n/2} \exp\left[ \frac{4 \pi^2 i t}{1+16\pi^2t^2} |x|^2 \right] \cdot \left( \sqrt{ \frac{1-4\pi i t}{1+4\pi i t}}\right)^{|\alpha|} \mathbf{h}_{\alpha} \left( \frac{x}{\sqrt{1+16\pi^2 t^2}} \right). \cr 
		\end{align*}      	
		Thus, we may write, whenever $f \in \mathcal{S}(\R^n),$ $f(x) = \sum_{\alpha \in \N^n} a_\alpha \mathbf{h}_{\alpha}(x)$,
		\begin{align}\label{eq:correspondence-HO-Schrodinger}
			e^{it\Delta} f(x) & = (1+4\pi i t)^{-n/2} \exp\left[ \frac{4 \pi^2 i t}{1+16\pi^2t^2} |x|^2 \right] \cr 
			& \times \sum_{\alpha \in \N^n} e^{i \arctan(-4 \pi t) |\alpha|} \cdot a_{\alpha} \cdot \mathbf{h}_{\alpha} \left( \frac{x}{\sqrt{1+16\pi^2 t^2}} \right)\cr  
			& = (1+16\pi^2 t^2)^{-n/4} \exp\left[ \frac{4 \pi^2 i t}{1+16\pi^2t^2} |x|^2 \right]  \cdot \Phi\left( \frac{x}{\sqrt{1+16\pi^2 t^2}} , \frac{\arctan(-4\pi t)}{4 \pi} \right). \cr 
		\end{align}
		The correspondence established in \eqref{eq:correspondence-HO-Schrodinger} above will be crucial for the next step. \\
		
		\noindent \textit{Step 2. Using the estimates by Escauriaza--Kenig--Ponce--Vega in order to deduce decay for the solution of the Schr\"odinger equation.} We make use of the translation from the previous step to establish the decay. We follow the overall approach of \cite{Escauriaza20081,Escauriaza2008,Escauriaza2006,Escauriaza2010,Escauriaza2016}. In particular, the proof of Theorem 1 in \cite{Escauriaza2010} yields as a by-product that, if $u$ is a solution to 
		\begin{align*}
			\begin{cases}
				i \partial_t u = - \Delta u & \text{ in } \, \R^n \times \R, \cr 
				u(x,0) = g(x) & \text{ on } \R^n, \cr
			\end{cases}
		\end{align*}
		then the function 
		$$v(x,t) = (it)^{-n/2} e^{-\frac{|x|^2}{4 i t}} \overline{u}(x/t,1/t -1)$$ 
		satisfies 
		\begin{align*}
			\begin{cases}
				i \partial_t v = - \Delta v & \text{ in } \, \R^n \times (0,+\infty), \cr 
				v(x,0) = (4 \pi)^{-n/2}  e^{\frac{-i|x|^2}{4}} \overline{\widehat{g}} (x/4 \pi) & \text{ on } \R^n, \cr
				v(x,1) =  i^{-n/2} e^{-|x|^2/4i} \overline{g}(x) & \text{ on } \R^n. \cr
			\end{cases}
		\end{align*}
		We shall use this fact with $g$ being a suitable dilation of $f.$ 
		
		Indeed, let $g(x) = f \left( \frac{x}{2 \sqrt{\pi}} \right).$ Then we know that $|g(x)| \lesssim e^{-\frac{a |x|^2}{4} }, \, |\widehat{g}(\xi/4\pi)| \lesssim e^{-\frac{a |x|^2}{4} },$ where we put $a = \tanh(2\alpha).$  
		
		For such $g,$ we have that the associated solution $v$ above satisfies $v(x,0), v(x,1)$  $ \, \, \, \in L^2(e^{\frac{a-\eps}{4} |x|^2} \, dx),$ for all $\eps > 0.$ We may now invoke the following result, which first appears in the works of Escauriaza--Kenig--Ponce--Vega \cite[Theorem~3]{Escauriaza2010sharp}.
		
		\begin{lemma}\label{lemma-optimal-gaussian-schrodinger} Assume that $w \in C([0,1], L^2(\R^n))$ satisfies 
			$$ i \partial_t w + \Delta w = 0, \, in \,\,\, \R^n \times [0,1].$$ 
			Then 
			$$ 
			\sup_{t \in [0,1]} \| e^{A(t)|x|^2} w(t)\|_2 \lesssim \|e^{A|x|^2} w(0)\|_2 + \|e^{A|x|^2} w(1)\|_2, 
			$$
			where 
			$$
			A(t) = \frac{R}{2(1+ R^2(2t -1)^2)}, \,\,\,\, A = \frac{R}{2(1+R^2)},
			$$
			and $0 < R < 1$.
		\end{lemma}
		
		We then use Lemma \ref{lemma-optimal-gaussian-schrodinger} with $A = \frac{a-\eps}{4}$. Let $R_{a, \eps}$ be the unique number between $0$ and $1$ such that $\frac{a-\eps}{4} = \frac{R_{a,\eps}}{2(1+R_{a,\eps}^2)}$, and denote the function $A(t)$ thus obtained by $A_{a,\eps}(t)$. We have then 
		$$
		\sup_{t \in [0,1]} \| e^{A_{a,\eps} (t)|x|^2} v(x,t) \|_{L^2(dx)} < + \infty. 
		$$
		Reverting back to $u,$ we find out that 
		$$ 
		\| e^{B_{a,\eps}(s)|x|^2} u(x,s) \|_2 < + \infty, \, \forall \,\, s > 0,
		$$
		where $B_{a,\eps}(s) =\frac{ A_{a,\eps}(1/(s+1)) }{(s+1)^2}. $  Observe that $e^{it\Delta} f(x) = u(2 \sqrt{\pi} x, 4 \pi t)$, and thus we have proven that 
		$$
		\| e^{4 \pi B_{a,\eps}(4 \pi t) |x|^2} e^{it \Delta} f(x)\|_2 < + \infty, \,\, \forall \,\, t>0. 
		$$
		\linebreak
		\noindent \textit{Step 3. Translating back.} Using the correspondence  \eqref{eq:correspondence-HO-Schrodinger} between solutions of the quantum harmonic oscillator and Schr\"odinger's equation, we see that 
		
		$$
		\left\| \exp((1+16 \pi^2 t^2) 4 \pi B_{a,\eps}(4 \pi t) |x|^2) \Phi\left(y, \frac{\arctan(-4\pi t)}{4 \pi} \right) \right\|_2 < + \infty, \, \forall \,\, t>0.
		$$
		Let 
		\begin{align*} 
			\Omega_{a,\eps}(t) := (1+16 \pi^2 t^2) 4 \pi \cdot B_{a,\eps}(4 \pi t) = \frac{(1+16\pi^2 t^2) \cdot 4 \pi \cdot R_{a,\eps}}{2[(4\pi t + 1)^2 +R_{a,\eps}^2(4\pi t -1)^2]}.
		\end{align*}
		Notice however that, as $R_{a,\eps} < 1,$ this function has \emph{exactly} one minimum point, which happens at $t = \frac{1}{4\pi},$ as 
		$$
		\Omega_{a,\eps}(t) - \pi R_{a,\eps} = \frac{\pi R_{a, \eps}(1-R_{a,\eps}^2)(4\pi t-1)^2}{[(4\pi t + 1)^2 +R_{a,\eps}^2(4\pi t -1)^2]}\ge 0.
		$$
		At $t = 1/4\pi,$ we have
		$$
		\Omega_{a,\eps}(1/4\pi) = \pi R_{a,\eps}.
		$$

		\begin{figure}
			\centering
			\begin{tikzpicture}[scale=3][line cap=round,line join=round,>=triangle 45,x=10cm,y=1cm]
				\begin{axis}[
					x=8.5cm,y=1.2cm,
					axis lines=middle,
					xmin=-0.06093005825143929,
					xmax=0.2681770426864066,
					ymin=-0.2067945424206617,
					ymax=1.754684936441864,
					yticklabels={,,},
					xticklabels={,,},]
					\clip(-0.06093005825143929,-1.567945424206617) rectangle (0.2881770426864066,2.054684936441864);
					\draw[line width=0.3pt,color=orange,smooth,samples=100,domain=-0.0005825143929:0.180] plot(\x,{1/4*((3.141592653589793*0.7+0.51*(4*3.141592653589793*(\x)-1)^(2)/(2*((4*3.141592653589793*(\x)+1)^(2)+0.49*(4*3.141592653589793*(\x)-1)^(2))))*(sign((\x))-sign((\x)-1/(4*3.141592653589793)))+(-sign((\x)-1/(2*3.141592653589793))+sign((\x)-1/(4*3.141592653589793)))*(3.141592653589793*0.7+0.51*(4*3.141592653589793*(1/(2*3.141592653589793)-(\x))-1)^(2)/(2*((4*3.141592653589793*(1/(2*3.141592653589793)-(\x))+1)^(2)+0.49*(4*3.141592653589793*(1/(2*3.141592653589793)-(\x))-1)^(2)))))});
					\draw (-0.029,0.91/0.7) node {\fontsize{1.6mm}{5.5cm}\selectfont $a \cdot \pi$};
					\draw (-0, 0.84/0.7) node {\LARGE $\cdot$};
					
					\draw[line width = 0.4 pt, dash pattern = on 1 pt off 1 pt, ] (0,0.768/0.7) -- (0.1585,0.768/0.7);
					\draw (0.18,0.91/0.7) node {\fontsize{1.6mm}{5.5cm}\selectfont $a \cdot \pi$};
					\draw (0.157, 0.84/0.7) node {\LARGE $\cdot$};
					\draw (0.079577,0.76/0.7) node {\LARGE $\cdot$};
					\draw (0.079577,0.68/0.7) node {\fontsize{1.6mm}{5.5cm}\selectfont $R_a \cdot \pi$ };
					\draw (-0.01,-0.135) node {\fontsize{1.6mm}{5.5cm}\selectfont $0$};
					\draw (0.079577,-0.135) node {\fontsize{1.6mm}{5.5cm}\selectfont $1/4 \pi$}; 
					\draw (0.172,-0.135) node {\fontsize{1.6mm}{5.5cm}\selectfont $1/2 \pi$};
					\draw[line width = 1pt] (0.1588,-0.05) -- (0.1588,0.05);
					\draw[line width = 1pt] (0.079577,-0.05) -- (0.079577,0.05);
					\draw (0.24,0.135) node {\fontsize{1.6mm}{5.5cm}\selectfont $t$};
				\end{axis}
			\end{tikzpicture}
			\caption{The limiting function $\Omega_a$ is bounded from below by the curve in \textcolor{orange}{orange} above. The dashed line represents the predicted gain of decay in Vemuri's conjecture.}
		\end{figure}
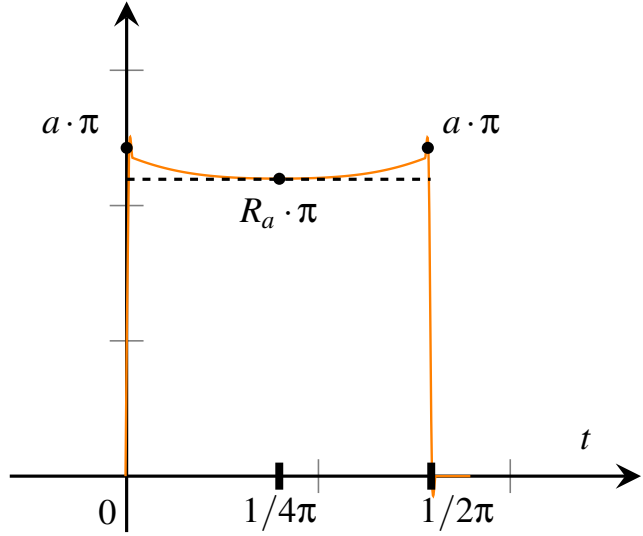

		As $\eps \to 0,$ we readily see that $R_{a,\eps} \to \tanh(\alpha).$ Let then $\Omega_a = \lim_{\eps \to 0} \Omega_{a,\eps}.$ We observe that, for $s \in (-1/16,0)$,  we have 
		$$
		\|\Phi(y,s) \cdot e^{b|y|^2} \|_2 < +\infty, \, \forall b < \Omega_a(-\tan(4 \pi s)).
		$$ 
		As $\Omega_a(-\tan(4 \pi s)) > \pi \tanh(\alpha)$ for $s \in (-1/16,0),$ there is $b(s) > \pi \tanh(\alpha)$ so that 
		\begin{equation}\label{eq:interpolated-gaussian-decay}
			\|\Phi(y,s) \cdot e^{b(s)|y|^2} \|_2 < +\infty.
		\end{equation}
		
		In order to extend this analysis to the rest of the claimed set, we notice that $|\Phi(y,-1/8)| = |\mathcal{F}f(y)|, |\Phi(y,1/8)| = |\mathcal{F}^{-1}f(y)|,$ and so on, so that $\Phi(y,k/8) \cdot \exp(a \pi |y|^2) \in L^{\infty}$ whenever $k \in \Z.$ Moreover, if we let $\Psi$ be a solution of \eqref{eq:quantum-harmonic} with the initial condition $\mathcal{F}^k(f),$ then we have 
		$$ |\Psi(y,t)| = |\Phi(y,t-k/8)|.$$ 
		Using these observations, together with the fact that \eqref{eq:quantum-harmonic} is \emph{time-reversible}, we are able to conclude that, for all $s \in \R\setminus \{\frac{1}{16} + \frac{k}{8}, \, k \in \Z\},$ there is $b(s) > \pi \tanh(\alpha)$ so that \eqref{eq:interpolated-gaussian-decay} holds. \\

		\noindent\textit{Step 4. Conclusion.} Finally, we use our main result in order to conclude. Indeed, by the explicit formula for the solution of \eqref{eq:quantum-harmonic}, we have that $|\mathcal{F}_y \Phi(y,t)| = |\Phi(y,t-1/8)|.$ If $t \not\in \{\frac{1}{16} + \frac{k}{8}, \, k \in \Z\},$ then $t - 1/8 \not\in \{\frac{1}{16} + \frac{k}{8}, \, k \in \Z\},$ and so, for $c(t) = \min\{b(t),b(t-1/8)\},$ we have 
		$$ \Phi(y,t) e^{c(t)|y|^2}, (\mathcal{F}_y\Phi(y,t)) \cdot e^{c(t)|y|^2} \in L^2(\R^n).$$
		From the main theorem, we have $\Phi(y,t) e^{(c(t)-\eps)|y|^2} \in L^{\infty}$ for any $\eps > 0.$ Taking $\eps > 0$ sufficiently small shows that 
		$$ \Phi(y,t) e^{\pi \tanh(\alpha) |y|^2} \in L^{\infty}(\R^n), \, \forall \, t \not\in \{\frac{1}{16} + \frac{k}{8}, \, k \in \Z\}.$$ 
		This finishes the proof of Theorem \ref{thm:vemuri}. 
	\end{proof}

	\begin{remark}
		Observe that the combination of the above lemmas in Section \ref{sec:lemmas} is quite powerful, but in order to use the   {Lemma \ref{lemma:decay}} to generate pointwise control, we do not need to impose that the function is controlled in space and frequency: much weaker estimates are more than enough to ensure the control that we need. This opens the door to understand other kinds of uncertainty principles in a broad range of situations, which we plan to do in future work. 
	\end{remark}
	
	\begin{remark}
		The proof of Theorem \ref{thm:vemuri} highlights that the original conjecture by Vemuri, albeit sharp if one considers the set of \emph{all} times $t \in \R,$ is almost-never sharp for a given time $t \in \R.$ Indeed, we can `upgrade' Vemuri's original conjecture to the following version: 
		
		\begin{conjecture}\label{conjecture:upgraded-vemuri} Let $f \in E_{\tanh(2\alpha)}^{\infty}(n),$ for some $\alpha > 0.$ Then 
			$$
			\Phi(\cdot, t) \in E_{\Omega_{\alpha}(t)}^{\infty}(n) \text{ for all } t \in \R,
			$$
			where we let $\Omega_{\alpha}(t) = \frac{(1+16\pi^2 s^2) \cdot 2 \cdot \tanh(\alpha)}{[(4\pi s + 1)^2 +\tanh(\alpha)^2(4\pi s -1)^2]},$ with $s = -\tan(4 \pi t).$
		\end{conjecture} 
		As we will see in the next section, in the particular cases of Theorem \ref{thm:second-laplace-full}, we are also able to settle this conjecture. This is a strong reason why we believe such a conjecture should be true.
	\end{remark}

	\section{On the endpoint versions of Corollary \ref{thm:main} and Theorem \ref{thm:vemuri}} 
	
	\subsection{Proof of Theorem \ref{thm:Laplace-sharp-1}}  \label{ssec:laplace-transforms-1} 
	
	In order to prove such a result, we start by noticing that, from Corollary \ref{thm:main}, we have that $\varphi \in E^{\infty}_{a-\varepsilon},$ for any $\varepsilon >0.$ We then have the following Lemma on decay of Laplace transforms: 
	
	\begin{lemma}\label{lemma-support-laplace} Let $\mu$ be a finite measure supported on the positive  real line. Suppose that its Laplace transform satisfies $|\mathcal{L} \mu(s)| \le C e^{-c_0 s}, \, \forall s >0,$ for some $C>0,$ $c_0 \in \mathbb{R}_+$. Then $\text{supp}(\mu) \subset [c_0, +\infty).$ 
	\end{lemma}
	
	\begin{proof} To prove this lemma, we define the function 
		$$
		F(z) = e^{-ic_0 z} \mathcal{L}\mu(-iz).
		$$
		Note that, by the definition of the Laplace transform, $F$ is a holomorphic function in the upper half plane $\mathbb{H}.$ Moreover, it has the following properties: 
		\begin{enumerate}
			\item \textit{$F$ is bounded on the real line.} This follows from the fact that $\mathcal{L}\mu(-it)$ is just a (rescaled) Fourer transform of the measure $\mu.$ As $\mu$ is finite,  its Fourier transform is bounded, and the modulation factor $e^{-ic_0 t}$ has absolute value one. 
			\item \textit{$|F(is)| \le C, \, \forall s > 0.$} This follows directly from our decay assumption. 
			\item \textit{$|F(z)| \le \tilde{C} e^{c_0|z|},$ for some $\tilde{C}> 0.$ } This follows again by the fact that $\mathcal{L}\mu$ is uniformly bounded on the upper half space. 
		\end{enumerate} 
		With these properties at hand, we are able to use the Phragm\'en--Lindel\"of principle in the first and second quadrants separately. This implies that $F$ is \emph{bounded and continuous} in $\mathbb{H}.$ 
		
		Thus, $F$ may be written as a Poisson integral of its boundary values. In particular, by the Young's convolution inequality, we have 
		\begin{align*}
			\| F(\cdot + iy) \|_{L^2(dx)} = \| (F|_{\R}) * P_y(x) \|_{L^2(dx)} \le \|F|_{\R} \|_2. 
		\end{align*}
		This inequality only holds, of course, if $F|_{\R} \in L^2.$ For now, let us assume that $d \mu = f(x) \, dx,$ with $f \in L^1 \cap L^2.$ Then the computation above shows us that $F \in H^2(\mathbb{H})$ (the Hardy space on the upper half space). In particular, by the Paley--Wiener theorem, we must have that $F|_{\R} = \widehat{h},$ where $\text{supp}(h) \subset (0,+\infty).$ 
		
		On the other hand, we see that $F|_{\R}$ may be written as a (rescaled) Fourier transform of $f(x+c_0) 1_{x+c_0 > 0}.$ Thus, $f(x+c_0)\ne0$ only if $x \ge 0,$ and so $f(y) \ne 0$ only if $y \ge c_0.$ This concludes the proof in the case where $d \mu = f(x) dx, \, f \in L^1 \cap L^2.$ \\
		
		For the general case, consider a smooth, positive compactly supported function $\phi$ so that $\text{supp}(\phi)$ is contained in the positive real line $(0,+\infty),$ and $\int \phi = 1.$ Let then $ \mu_{\eps}(x) = (d\mu) * \phi_{\eps} (x),$ where we define $\phi_{\eps}(y) = \frac{1}{\eps} \phi\left( \frac{y}{\eps}\right).$ 
		
		By the Young's inequality, we have $\mu_{\eps} \in L^1 \cap L^2(\R).$ Moreover, $\mathcal{L} \mu_{\eps} = (\mathcal{L} \mu) \cdot (\mathcal{L} \phi_{\eps})$ by the definition of $\phi.$ As $|\mathcal{L} f(s)| \le \|f\|_1$ uniformly on $s \in \{z \in \C \colon \text{Re}(z) \ge 0\},$ we have that 
		$$
		|\mathcal{L} \mu_{\eps} (s)| \le C e^{-c_0 s},
		$$
		uniformly on $\eps > 0.$ Thus, $\text{supp}(\mu_{\eps}) \subset [c_0,+\infty).$ As $\mu_{\eps} \stackrel{\ast}{\rightharpoonup} d \mu$ in the space of finite measures on the real line, we see that $\text{supp}(\mu) \subset [c_0,+\infty),$ as desired. 
	\end{proof}
	
	We are now ready to finish the proof of Theorem \ref{thm:Laplace-sharp-1}. 
	
	\begin{proof}[Proof of Theorem \ref{thm:Laplace-sharp-1}] First we notice that, by our main result, $\varphi \in E^{\infty}_{a-\eps}, \, \forall \eps> 0.$ This implies that $|\mathcal{L}\mu(s)| \lesssim_{\eps} e^{-(a-\eps)s},\, \forall s >0.$ By Lemma \ref{lemma-support-laplace}, we conclude that $\text{supp}(\mu) \subset [a-\eps,+\infty), \, \, \forall \eps > 0.$ This plainly implies that $\supp(\mu) \subset [a,+\infty), $ which in turn implies that $|\varphi(x)| \lesssim e^{-\pi a |x|^2}.$ By observing that the Fourier transform 
		$$
		\widehat{\varphi}(\xi) = \int_0^{\infty} \frac{1}{t^{1/2}} e^{-\frac{\pi}{t} |\xi|^2} \, \, d \mu(t) =: \int_0^{\infty} e^{-r \pi |\xi|^2} d\nu(r)
		$$
		also satisfies that $|\nu|(\R_+) \le \int_a^{\infty} \frac{1}{t^{1/2}} \, |d \mu|(t) < + \infty,$  we may employ the same reasoning to conclude that also $|\widehat{\varphi}(\xi)| \lesssim e^{-a\pi|\xi|^2}.$ This finishes the proof.
	\end{proof}
	
	The above results lead us to the following question: \emph{if $f\in E^{2}_{a}$, does it then follow that $f\in E^{\infty}_{a}$?} We are led to speculate that such a question has an affirmative answer based on the previous theorem, and thus we believe that we should have a control in the sub-critical regime of the uncertainty principles \emph{without} the $\epsilon$-loss that is present in the  {Theorem \ref{thm:main}}. 
	
	In view of the considerations above, one may may wonder whether the class of Laplace transforms presented in Subsection \ref{ssec:laplace-transforms-1} represents the almost sharp rate of decay obtained in Theorem \ref{thm:vemuri}. In order to analyse that, we need to introduce the following concept. 
	
	For $\beta \in \R,$ we define the \emph{fractional Fourier transform of order $\beta$} to act on the Hermite functions as 
	$$ 
	\mathcal{F}_{\beta} h_k = e^{-ik\beta} h_k,
	$$
	and extended it to $L^2$ in the canonical way. By the properties of the Hermite polynomials and the Mehler kernel \cite{Beckner}, one is led to deduce that these transforms have the following representation as \emph{integral} transforms: 
	\begin{equation}\label{eq:FrFT}
		\mathcal{F}_{\beta} f(x) = \frac{e^{i (\theta(\beta) \pi/2 - \beta/2)}}{\sqrt{|\sin(\beta)|}} e^{i\pi x^2 \cot(\beta)} \int_{\R} e^{-2 \pi i ( xy \csc(\beta) - y^2 \cot(\beta)/2)} f(y) \, dy,
	\end{equation}
	where $\theta(\beta) = \text{sign}(\sin (\beta)).$  The relationship between fractional Fourier transforms and the evolution of the quantum harmonic oscillator is evident from the definition. Indeed, we may write 
	$$
	\Phi(y,t) = e^{2 \pi i t} (\mathcal{F}_{-4 \pi t} f)(y). 
	$$
	The key feature of this definition is the relationship with \eqref{eq:FrFT}, which allows us to compute easily fractional Fourier transforms of Gaussians and related functions. 
	
	As a first observation, notice that for $f(x) = e^{-\lambda \pi |x|^2},$ the integral in \eqref{eq:FrFT} is just the Fourier transform of $e^{-\pi y^2(\lambda + i \cot(\beta))}$ evaluated at the point $x \csc(\beta).$ This in turn evaluates directly to 
	\begin{equation}\label{eq:gaussian-fractional} 
		\frac{1}{(\lambda + i\cot(\beta))^{1/2}} e^{- \frac{\pi \csc^2(\beta) |x|^2}{\lambda + i \cot(\beta)} }= \frac{1}{(\lambda + i\cot(\beta))^{1/2}}  e^{ - \pi |x|^2 \frac{\lambda \csc^2(\beta) - i \csc^2(\beta)\cot(\beta) }{\lambda^2 + \cot^2(\beta)}}.
	\end{equation}
	Now let $\varphi(x) = \mathcal{L}\mu(\pi |x|^2),$ with $d\mu$ a finite measure. If we have $\varphi \in E^{\infty}_a,$ then Theorem \ref{thm:Laplace-sharp-1} tells us that $\supp(\mu) \subset [a,1/a].$ The computation of the fractional Fourier transform of the Gaussian above shows that 
	\begin{align}\label{eq:bound-FrFT-Laplace} 
		|\mathcal{F}_{\beta} \varphi (x)| & \lesssim \int_a^{1/a} e^{ - \pi |x|^2 \frac{\lambda \csc^2(\beta)  }{\lambda^2 + \cot^2(\beta)}} \, |d \mu|(\lambda) \cr 
		& \lesssim \|\mu\|_{TV} \max\{ e^{ - \pi |x|^2 \frac{a \csc^2(\beta)  }{a^2 + \cot^2(\beta)}},  e^{ - \pi |x|^2 \frac{a \csc^2(\beta)  }{1 + a\cot^2(\beta)}} \}. \cr
	\end{align}
	For $a < 1,$ we can see that 
	$$ 
	\min\left\{ \frac{a \csc^2(\beta)  }{a^2 + \cot^2(\beta)}, \frac{a \csc^2(\beta)  }{1 + a\cot^2(\beta)}\right\} \ge a,
	$$
	with equality if and only if $\beta = k\pi/2,\,\, k \in \Z.$ Thus, \eqref{eq:bound-FrFT-Laplace} implies that there is $\vartheta(\beta) > a$ whenever $ \beta \neq k\pi/2, \,\, k \in \Z,$ so that  $$|\mathcal{F}_{\beta}\varphi(x)| \lesssim \|\mu\|_{TV} e^{- \pi \vartheta(\beta)|x|^2}.$$ 
	Thus, by relating the fractional Fourier transform to the quantum harmonic oscillator, we obtain much stronger version of Vemuri's conjecture when $\varphi(x) = \mathcal{L}\mu(\pi |x|^2).$

	\subsection{Proof of Theorem \ref{thm:second-laplace-full}} 
	
	The proof of Theorem \ref{thm:second-laplace-full} is based on the following result on support of measures on the circle with rapidly decaying Laplace transform. This, on the other hand, is an analogue  of Lemma \ref{lemma-support-laplace} on the circle. 
	
	\begin{lemma}\label{lemma-support-measure-circle} Let $\nu$ be a finite measure on the \emph{circle} $\mathbb{S}^1\subset \C.$ Suppose that the Laplace transform 
		$$ 
		\mathcal{L}\nu(t) = \int_{\mathbb{S}^1} e^{-z t} \, d \nu(z)
		$$
		satisfies $|\mathcal{L}\nu(t)| \le C e^{-c_0 t}$ for some $C>0$, $c_0 \in \R_+.$ Suppose additionally that there is $\delta > 0$ so that $\supp(\nu) \subset \mathbb{S}^1 \cap \{z \colon \text{Re}(z) > -1 + \delta\}$  (or, equivalently, $-1\notin \supp (\nu)$). Then 
		$$ \supp(\nu) \subset \mathbb{S}^1 \cap \{z \colon \text{Re}(z) \ge c_0\}. $$
	\end{lemma}
	
	\begin{proof} As $\mathcal{L}\nu$ is defined and decays as $e^{-c_0 t}$ on the positive half-line, we may take its (real line) Laplace transform, which we will denote by 
		$$
		F(s) = \int_0^{\infty} e^{st} \mathcal{L}\nu(t) \, dt. 
		$$
		By the decay of $\mathcal{L}\nu,$ $F$ is well-defined and holomorphic on the half-space $\{s \colon \text{Re}(s) < c_0 \}.$ Moreover, $F$ obeys the bound 
		\begin{equation}\label{eq:first-bound-F}
			|F(s)| \le C |c_0 - \text{Re}(s)|^{-1}, \,\, \text{ whenever } \text{Re}(s) < c_0.
		\end{equation}
		On the other hand, by Fubini's Theorem, we have the representation 
		\begin{equation}\label{eq:Cauchy-transform} 
			F(s) = \int_{\mathbb{S}^1} \frac{1}{z-s} \, d \nu(z), \,\, \text{ whenever } |s| < 1.
		\end{equation}
		The right-hand side of \eqref{eq:Cauchy-transform} can be further extended as an analytic function whenever $s \not \in \supp(\nu)$, as the set $\C \setminus \supp(\nu)$ is \emph{connected}  thanks to the additional hypothesis on the support of $\nu$. Thus, $\int_{\mathbb{S}^1} \frac{1}{z-s} \, d \nu(z)$ must agree with $F(s)$ on the intersection between $\mathbb{C} \setminus \supp(\nu)$ and $\{\text{Re}(s) < c_0\}.$ We will also denote by $F(s)$ the analytic function that continues over the union of both sets above. 
		
		Notice that, by this definition, we also have 
		\begin{equation}\label{eq:second-bound-F}
			|F(s)| \lesssim \text{dist}(s,\supp(\nu))^{-1}.
		\end{equation}
		
		In order to finish, we observe that we may replace the measure $\nu$ by $\tilde{\nu} = \nu|_A,$ where $A = \{ z \in \mathbb{S}^1 \colon \text{Re}(z) \le c_0 \},$ in each of the steps above. Let $F_{\tilde{\nu}}$ be the function constructed in association with it. Then:
		\begin{enumerate}
			\item $F_{\tilde{\nu}}$ is well-defined and holomorphic on $\C\setminus\left\{c_0 \pm i \sqrt{1-c_0^2}\right\};$ 
			\item $|F_{\tilde{\nu}}(s)| \lesssim \left|s-(c_0 + i\sqrt{1-c_0^2})\right|^{-1} + \left|s-(c_0 - i\sqrt{1-c_0^2})\right|^{-1}, \, \forall s \in \C.$ \\
			
			Indeed, If $\text{Re}(s) \le 0,$ then the claim is trivial in light of \eqref{eq:first-bound-F}. More generally, the claim follows by either \eqref{eq:first-bound-F}, \eqref{eq:Cauchy-transform} or \eqref{eq:second-bound-F} whenever 
			$$\text{dist}(s,\supp({\tilde{\nu}})) \ge B,$$
			with $B>0$ an absolute constant to be determined later. Thus, we may restrict ourselves to $\text{Re}(s) > 0, \text{dist}(s,\supp{\tilde{\nu}}) < B.$ Let $c_0 + i\sqrt{1-c_0^2} = z_0,$ for shortness. 
			
			Consider first the region $R_1 = \{ s = z_0 + w, \text{Re}(s) > 0, \text{Im}(w) > \frac{(c_0+1)}{2} |w| \}.$ In that region, the angle between $w$ and $z_0$ is always \emph{strictly} less than $\pi/2,$ and thus we have 
			$$ |s|^2 = 1 + |w|^2 + 2 \langle z_0, w \rangle \ge 1 + C(c_0) |w|,$$
			where we may write, in more explicit terms, 
			$$C(c_0) = \sqrt{1-c_0^2} \frac{c_0 + 1}{2} - \left( 1 - \frac{(c_0+1)^2}{4}\right)^{1/2} c_0 \ge \sqrt{1-c_0^2} \frac{1-c_0}{2}.$$
			Therefore, $|s|-1 \gtrsim |w| = |s - z_0|,$ and as $|s|-1 = \text{dist}(s,\mathbb{S}^1)$ for $s \in R_1,$ we have the claim in that region from \eqref{eq:second-bound-F}. Analogously, if we consider the region 
			$$R_2 = \left\{ s = z_0 + w, \text{Re}(s) > 0, \text{Im}(w) <- \frac{c_0+1}{2} |w|, |w| \ll 1 \right\},$$
			we see that $|s|^2 \le 1 - \kappa(c_0) |w|,$ and thus $\text{dist}(s,\mathbb{S}^1) = 1 - |s| \gtrsim |w|,$ and the conclusion follows in the same manner. 
			
			Now, if we let
			$$R_3 = \left\{ s = z_0 + w, c_0 > \text{Re}(s) > 0, \text{Im}(w) \in \left(-\frac{c_0+1}{2}|w|, \frac{c_0+1}{2}|w|\right)\right\},$$ 
			we have $|\text{Re}(w)| > \left(1 - \frac{(c_0+1)^2}{4}\right)^{1/2} |w|.$ In particular, $|\text{Re}(s)-c_0| = |\text{Re}(w)| \gtrsim |w|,$ and \eqref{eq:first-bound-F} gives us the result once again. On the other hand, the estimate in the region $R_4 = \left\{ s = z_0 + w,  \text{Re}(s) > c_0,  \text{Im}(w) \in \left(-\frac{c_0+1}{2}|w|, \frac{c_0+1}{2}|w|\right)\right\}$ follows directly from \eqref{eq:second-bound-F} and the fact that $\supp(\tilde{\nu}) \subset \mathbb{S}^1 \cap \{\text{Re}(s) \le c_0 \}.$ 
			
			By repeating the same process above, but reflected, to the point $\overline{z_0}$ shows the result in a neighbourhood of size $\delta(c_0) > 0$ of $\text{supp}(\tilde{\nu}).$ Let then $B = \delta(c_0)$ in the beginning. This proves the claim.

			\definecolor{qqqqff}{rgb}{0,0,1}
			\definecolor{zzffqq}{rgb}{0.6,1,0}
			\definecolor{ffffzz}{rgb}{1,1,0.6}
			\definecolor{ffqqqq}{rgb}{1,0,0}
			\definecolor{uuuuuu}{rgb}{0.26666666666666666,0.26666666666666666,0.26666666666666666}
			
			\begin{figure} 
				\centering
				\begin{tikzpicture}[scale=1.5][line cap=round,line join=round,>=triangle 45,x=1cm,y=1cm]
					\begin{axis}[
						x=1cm,y=1cm,
						axis lines=middle,
						xmin=-2.106171796401762,
						xmax=3.06926886601463,
						ymin=-2.383137726113157,
						ymax=3.2561395994443574,
						xtick={-2,-1.5,...,3},
						ytick={-2,-1.5,...,3},yticklabels={,,},
						xticklabels={,,},]
						\clip(-2.106171796401762,-2.383137726113157) rectangle (3.06926886601463,3.2561395994443574);
						\draw [line width=0.5pt] (0,0) circle (1cm);
						\draw [line width=0.5pt] (0.7,-2.383137726113157) -- (0.7,3.2561395994443574);
						\draw[line width=0pt,dash pattern=on 1pt off 1pt,color=ffqqqq,fill=ffqqqq,fill opacity=0.25](0,3.2561395994443574)--(0,2.5321)--(0.0019454938612041542,2.527035879479286)--(0.02279779998925762,2.4727573266279625)--(0.04365010611731155,2.4184787737766387)--(0.06450241224536456,2.3642002209253157)--(0.08535471837341801,2.3099216680739922)--(0.10620702450147149,2.2556431152226697)--(0.12705933062952496,2.201364562371346)--(0.1479116367575784,2.1470860095200237)--(0.16876394288563187,2.0928074566687)--(0.18961624901368535,2.0385289038173773)--(0.21046855514173882,1.9842503509660536)--(0.2313208612697923,1.929971798114731)--(0.25217316739784573,1.8756932452634074)--(0.27302547352589923,1.8214146924120838)--(0.2938777796539527,1.7671361395607612)--(0.3147300857820061,1.7128575867094373)--(0.3355823919100596,1.658579033858115)--(0.35643469803811306,1.6043004810067911)--(0.37728700416616656,1.5500219281554686)--(0.39813931029422,1.495743375304145)--(0.41899161642227345,1.4414648224528224)--(0.43984392255032695,1.3871862696014987)--(0.4606962286783804,1.3329077167501762)--(0.48154853480643384,1.2786291638988525)--(0.5024008409344873,1.2243506110475288)--(0.5232531470625408,1.1700720581962063)--(0.5441054531905942,1.1157935053448824)--(0.5649577593186478,1.0615149524935599)--(0.5858100654467012,1.0072363996422364)--(0.6066623715747547,0.9529578467909139)--(0.6275146777028081,0.89867929393959)--(0.6483669838308616,0.8444007410882675)--(0.6692192899589151,0.7901221882369436)--(0.6900715960869686,0.735843635385621)--(0.7004977491509953,0.711295641040041)--(0.7031042874170019,0.7180804601464558)--(0.7057108256830086,0.7248652792528711)--(0.710923902215022,0.7384349174657022)--(0.7317762083430754,0.7927134703170261)--(0.7526285144711289,0.8469920231683487)--(0.7734808205991823,0.9012705760196721)--(0.7943331267272359,0.9555491288709946)--(0.8151854328552893,1.0098276817223186)--(0.8360377389833428,1.0641062345736412)--(0.8568900451113962,1.1183847874249646)--(0.8777423512394497,1.1726633402762872)--(0.8985946573675032,1.226941893127611)--(0.9194469634955567,1.2812204459789336)--(0.9402992696236101,1.3354989988302572)--(0.9611515757516635,1.389777551681581)--(0.982003881879717,1.4440561045329035)--(1.0028561880077704,1.498334657384227)--(1.0237084941358239,1.5526132102355497)--(1.0445608002638775,1.6068917630868733)--(1.065413106391931,1.6611703159381959)--(1.0862654125199844,1.7154488687895195)--(1.1071177186480379,1.769727421640842)--(1.1279700247760913,1.824005974492166)--(1.1488223309041448,1.8782845273434896)--(1.1696746370321982,1.9325630801948122)--(1.1905269431602516,1.9868416330461358)--(1.211379249288305,2.0411201858974586)--(1.2322315554163585,2.095398738748782)--(1.253083861544412,2.1496772916001046)--(1.2739361676724656,2.2039558444514284)--(1.294788473800519,2.2582343973027506)--(1.3156407799285725,2.3125129501540744)--(1.336493086056626,2.366791503005397)--(1.3573453921846794,2.421070055856721)--(1.3781976983127329,2.4753486087080443)--(1.3990500044407863,2.5296271615593673)--(1.4199023105688398,2.5839057144106907)--(1.4407546166968932,2.6381842672620133)--(1.4616069228249466,2.692462820113337)--(1.4824592289530003,2.7467413729646593)--(1.5033115350810538,2.801019925815983)--(1.5241638412091072,2.8552984786673057)--(1.5450161473371606,2.9095770315186296)--(1.565868453465214,2.963855584369952)--(1.5867207595932675,3.0181341372212755)--(1.607573065721321,3.0724126900725994)--(1.6284253718493744,3.126691242923922)--(1.6492776779774279,3.1809697957752454)--(1.6701299841054813,3.235248348626568)--(1.6781558199939908,3.2561395994443574);
						\draw[line width=0pt,dash pattern=on 1pt off 1pt,color=ffqqqq,fill=ffqqqq,fill opacity=0.25](0,3.2561395994443574)--(0,2.5321)--(0.0019454938612041542,2.527035879479286)--(0.02279779998925762,2.4727573266279625)--(0.04365010611731155,2.4184787737766387)--(0.06450241224536456,2.3642002209253157)--(0.08535471837341801,2.3099216680739922)--(0.10620702450147149,2.2556431152226697)--(0.12705933062952496,2.201364562371346)--(0.1479116367575784,2.1470860095200237)--(0.16876394288563187,2.0928074566687)--(0.18961624901368535,2.0385289038173773)--(0.21046855514173882,1.9842503509660536)--(0.2313208612697923,1.929971798114731)--(0.25217316739784573,1.8756932452634074)--(0.27302547352589923,1.8214146924120838)--(0.2938777796539527,1.7671361395607612)--(0.3147300857820061,1.7128575867094373)--(0.3355823919100596,1.658579033858115)--(0.35643469803811306,1.6043004810067911)--(0.37728700416616656,1.5500219281554686)--(0.39813931029422,1.495743375304145)--(0.41899161642227345,1.4414648224528224)--(0.43984392255032695,1.3871862696014987)--(0.4606962286783804,1.3329077167501762)--(0.48154853480643384,1.2786291638988525)--(0.5024008409344873,1.2243506110475288)--(0.5232531470625408,1.1700720581962063)--(0.5441054531905942,1.1157935053448824)--(0.5649577593186478,1.0615149524935599)--(0.5858100654467012,1.0072363996422364)--(0.6066623715747547,0.9529578467909139)--(0.6275146777028081,0.89867929393959)--(0.6483669838308616,0.8444007410882675)--(0.6692192899589151,0.7901221882369436)--(0.6900715960869686,0.735843635385621)--(0.7004977491509953,0.711295641040041)--(0.7031042874170019,0.7180804601464558)--(0.7057108256830086,0.7248652792528711)--(0.710923902215022,0.7384349174657022)--(0.7317762083430754,0.7927134703170261)--(0.7526285144711289,0.8469920231683487)--(0.7734808205991823,0.9012705760196721)--(0.7943331267272359,0.9555491288709946)--(0.8151854328552893,1.0098276817223186)--(0.8360377389833428,1.0641062345736412)--(0.8568900451113962,1.1183847874249646)--(0.8777423512394497,1.1726633402762872)--(0.8985946573675032,1.226941893127611)--(0.9194469634955567,1.2812204459789336)--(0.9402992696236101,1.3354989988302572)--(0.9611515757516635,1.389777551681581)--(0.982003881879717,1.4440561045329035)--(1.0028561880077704,1.498334657384227)--(1.0237084941358239,1.5526132102355497)--(1.0445608002638775,1.6068917630868733)--(1.065413106391931,1.6611703159381959)--(1.0862654125199844,1.7154488687895195)--(1.1071177186480379,1.769727421640842)--(1.1279700247760913,1.824005974492166)--(1.1488223309041448,1.8782845273434896)--(1.1696746370321982,1.9325630801948122)--(1.1905269431602516,1.9868416330461358)--(1.211379249288305,2.0411201858974586)--(1.2322315554163585,2.095398738748782)--(1.253083861544412,2.1496772916001046)--(1.2739361676724656,2.2039558444514284)--(1.294788473800519,2.2582343973027506)--(1.3156407799285725,2.3125129501540744)--(1.336493086056626,2.366791503005397)--(1.3573453921846794,2.421070055856721)--(1.3781976983127329,2.4753486087080443)--(1.3990500044407863,2.5296271615593673)--(1.4199023105688398,2.5839057144106907)--(1.4407546166968932,2.6381842672620133)--(1.4616069228249466,2.692462820113337)--(1.4824592289530003,2.7467413729646593)--(1.5033115350810538,2.801019925815983)--(1.5241638412091072,2.8552984786673057)--(1.5450161473371606,2.9095770315186296)--(1.565868453465214,2.963855584369952)--(1.5867207595932675,3.0181341372212755)--(1.607573065721321,3.0724126900725994)--(1.6284253718493744,3.126691242923922)--(1.6492776779774279,3.1809697957752454)--(1.6701299841054813,3.235248348626568)--(1.6781558199939908,3.2561395994443574);
						\draw[line width=0pt,dash pattern=on 1pt off 1pt,color=ffffzz,fill=ffffzz,fill opacity=0.25](0.7004977491509953,0.7087043589599589)--(0.6900715960869686,0.6841563646143788)--(0.6692192899589151,0.6298778117630562)--(0.6483669838308616,0.5755992589117324)--(0.6275146777028081,0.5213207060604098)--(0.6066623715747547,0.4670421532090864)--(0.5858100654467012,0.4127636003577638)--(0.5649577593186478,0.35848504750643995)--(0.5441054531905942,0.3042064946551174)--(0.5232531470625408,0.2499279418037935)--(0.5024008409344873,0.19564938895247097)--(0.48154853480643384,0.1413708361011475)--(0.4606962286783804,0.0870922832498236)--(0.43984392255032695,0.03281373039850108)--(0.41899161642227345,-0.021464822452822375)--(0.39813931029422,-0.0757433753041449)--(0.37728700416616656,-0.1300219281554688)--(0.35643469803811306,-0.18430048100679133)--(0.3355823919100596,-0.2385790338581148)--(0.3147300857820061,-0.2928575867094373)--(0.2938777796539527,-0.3471361395607612)--(0.27302547352589923,-0.4014146924120838)--(0.25217316739784573,-0.45569324526340765)--(0.2313208612697923,-0.5099717981147311)--(0.21046855514173882,-0.5642503509660536)--(0.18961624901368535,-0.6185289038173771)--(0.16876394288563187,-0.6728074566686996)--(0.1479116367575784,-0.7270860095200236)--(0.12705933062952496,-0.781364562371346)--(0.10620702450147149,-0.83564311522267)--(0.08535471837341801,-0.8899216680739925)--(0.06450241224536456,-0.944200220925316)--(0.04365010611731155,-0.9984787737766385)--(0.02279779998925762,-1.052757326627962)--(0.0019454938612041542,-1.1070358794792863)--(0,-1.1121)--(0,-2.383137726113157)--(1.8882972439927612,-2.383137726113157)--(1.878653045386016,-2.3580338771397997)--(1.8578007392579625,-2.303755324288477)--(1.836948433129909,-2.249476771437155)--(1.8160961270018552,-2.1951982185858294)--(1.7952438208738022,-2.140919665734507)--(1.7743915147457487,-2.0866411128831843)--(1.7535392086176953,-2.03236256003186)--(1.7326869024896419,-1.9780840071805377)--(1.7118345963615884,-1.9238054543292142)--(1.690982290233535,-1.8695269014778917)--(1.6701299841054813,-1.8152483486265683)--(1.6492776779774279,-1.7609697957752457)--(1.6284253718493744,-1.7066912429239223)--(1.607573065721321,-1.6524126900725997)--(1.5867207595932675,-1.5981341372212754)--(1.565868453465214,-1.543855584369952)--(1.5450161473371606,-1.4895770315186294)--(1.5241638412091072,-1.435298478667306)--(1.5033115350810538,-1.3810199258159834)--(1.4824592289530003,-1.3267413729646589)--(1.4616069228249466,-1.2724628201133366)--(1.4407546166968932,-1.2181842672620131)--(1.4199023105688398,-1.1639057144106906)--(1.3990500044407863,-1.1096271615593671)--(1.3781976983127329,-1.0553486087080446)--(1.3573453921846794,-1.0010700558567212)--(1.336493086056626,-0.9467915030053967)--(1.3156407799285725,-0.8925129501540746)--(1.294788473800519,-0.8382343973027507)--(1.2739361676724656,-0.7839558444514282)--(1.253083861544412,-0.7296772916001047)--(1.2322315554163585,-0.6753987387487822)--(1.211379249288305,-0.6211201858974583)--(1.1905269431602516,-0.5668416330461358)--(1.1696746370321982,-0.5125630801948123)--(1.1488223309041448,-0.4582845273434898)--(1.1279700247760913,-0.4040059744921659)--(1.1071177186480379,-0.349727421640842)--(1.0862654125199844,-0.2954488687895195)--(1.065413106391931,-0.24117031593819602)--(1.0445608002638775,-0.18689176308687347)--(1.0237084941358239,-0.13261321023554956)--(1.0028561880077704,-0.07833465738422704)--(0.982003881879717,-0.02405610453290359)--(0.9611515757516635,0.030222448318418933)--(0.9402992696236101,0.08450100116974285)--(0.9194469634955567,0.1387795540210663)--(0.8985946573675032,0.19305810687238883)--(0.8777423512394497,0.24733665972371274)--(0.8568900451113962,0.30161521257503526)--(0.8360377389833428,0.3558937654263587)--(0.8151854328552893,0.41017231827768125)--(0.7943331267272359,0.4644508711290051)--(0.7734808205991823,0.5187294239803276)--(0.7526285144711289,0.5730079768316516)--(0.7317762083430754,0.6272865296829742)--(0.710923902215022,0.6815650825342976)--(0.7057108256830086,0.6951347207471287)--(0.7031042874170019,0.7019195398535439)--(0.7004977491509953,0.7087043589599589);
						\draw[line width=0pt,dash pattern=on 1pt off 1pt,color=ffffzz,fill=ffffzz,fill opacity=0.25](0.7004977491509953,0.7087043589599589)--(0.6900715960869686,0.6841563646143788)--(0.6692192899589151,0.6298778117630562)--(0.6483669838308616,0.5755992589117324)--(0.6275146777028081,0.5213207060604098)--(0.6066623715747547,0.4670421532090864)--(0.5858100654467012,0.4127636003577638)--(0.5649577593186478,0.35848504750643995)--(0.5441054531905942,0.3042064946551174)--(0.5232531470625408,0.2499279418037935)--(0.5024008409344873,0.19564938895247097)--(0.48154853480643384,0.1413708361011475)--(0.4606962286783804,0.0870922832498236)--(0.43984392255032695,0.03281373039850108)--(0.41899161642227345,-0.021464822452822375)--(0.39813931029422,-0.0757433753041449)--(0.37728700416616656,-0.1300219281554688)--(0.35643469803811306,-0.18430048100679133)--(0.3355823919100596,-0.2385790338581148)--(0.3147300857820061,-0.2928575867094373)--(0.2938777796539527,-0.3471361395607612)--(0.27302547352589923,-0.4014146924120838)--(0.25217316739784573,-0.45569324526340765)--(0.2313208612697923,-0.5099717981147311)--(0.21046855514173882,-0.5642503509660536)--(0.18961624901368535,-0.6185289038173771)--(0.16876394288563187,-0.6728074566686996)--(0.1479116367575784,-0.7270860095200236)--(0.12705933062952496,-0.781364562371346)--(0.10620702450147149,-0.83564311522267)--(0.08535471837341801,-0.8899216680739925)--(0.06450241224536456,-0.944200220925316)--(0.04365010611731155,-0.9984787737766385)--(0.02279779998925762,-1.052757326627962)--(0.0019454938612041542,-1.1070358794792863)--(0,-1.1121)--(0,-2.383137726113157)--(1.8882972439927612,-2.383137726113157)--(1.878653045386016,-2.3580338771397997)--(1.8578007392579625,-2.303755324288477)--(1.836948433129909,-2.249476771437155)--(1.8160961270018552,-2.1951982185858294)--(1.7952438208738022,-2.140919665734507)--(1.7743915147457487,-2.0866411128831843)--(1.7535392086176953,-2.03236256003186)--(1.7326869024896419,-1.9780840071805377)--(1.7118345963615884,-1.9238054543292142)--(1.690982290233535,-1.8695269014778917)--(1.6701299841054813,-1.8152483486265683)--(1.6492776779774279,-1.7609697957752457)--(1.6284253718493744,-1.7066912429239223)--(1.607573065721321,-1.6524126900725997)--(1.5867207595932675,-1.5981341372212754)--(1.565868453465214,-1.543855584369952)--(1.5450161473371606,-1.4895770315186294)--(1.5241638412091072,-1.435298478667306)--(1.5033115350810538,-1.3810199258159834)--(1.4824592289530003,-1.3267413729646589)--(1.4616069228249466,-1.2724628201133366)--(1.4407546166968932,-1.2181842672620131)--(1.4199023105688398,-1.1639057144106906)--(1.3990500044407863,-1.1096271615593671)--(1.3781976983127329,-1.0553486087080446)--(1.3573453921846794,-1.0010700558567212)--(1.336493086056626,-0.9467915030053967)--(1.3156407799285725,-0.8925129501540746)--(1.294788473800519,-0.8382343973027507)--(1.2739361676724656,-0.7839558444514282)--(1.253083861544412,-0.7296772916001047)--(1.2322315554163585,-0.6753987387487822)--(1.211379249288305,-0.6211201858974583)--(1.1905269431602516,-0.5668416330461358)--(1.1696746370321982,-0.5125630801948123)--(1.1488223309041448,-0.4582845273434898)--(1.1279700247760913,-0.4040059744921659)--(1.1071177186480379,-0.349727421640842)--(1.0862654125199844,-0.2954488687895195)--(1.065413106391931,-0.24117031593819602)--(1.0445608002638775,-0.18689176308687347)--(1.0237084941358239,-0.13261321023554956)--(1.0028561880077704,-0.07833465738422704)--(0.982003881879717,-0.02405610453290359)--(0.9611515757516635,0.030222448318418933)--(0.9402992696236101,0.08450100116974285)--(0.9194469634955567,0.1387795540210663)--(0.8985946573675032,0.19305810687238883)--(0.8777423512394497,0.24733665972371274)--(0.8568900451113962,0.30161521257503526)--(0.8360377389833428,0.3558937654263587)--(0.8151854328552893,0.41017231827768125)--(0.7943331267272359,0.4644508711290051)--(0.7734808205991823,0.5187294239803276)--(0.7526285144711289,0.5730079768316516)--(0.7317762083430754,0.6272865296829742)--(0.710923902215022,0.6815650825342976)--(0.7057108256830086,0.6951347207471287)--(0.7031042874170019,0.7019195398535439)--(0.7004977491509953,0.7087043589599589);
						\draw[line width=0pt,dash pattern=on 1pt off 1pt,color=zzffqq,fill=zzffqq,fill opacity=0.25](0,2.5321)--(0,-1.1121)--(0.0019454938612041542,-1.1070358794792863)--(0.02279779998925762,-1.052757326627962)--(0.04365010611731155,-0.9984787737766385)--(0.06450241224536456,-0.944200220925316)--(0.08535471837341801,-0.8899216680739925)--(0.10620702450147149,-0.83564311522267)--(0.12705933062952496,-0.781364562371346)--(0.1479116367575784,-0.7270860095200236)--(0.16876394288563187,-0.6728074566686996)--(0.18961624901368535,-0.6185289038173771)--(0.21046855514173882,-0.5642503509660536)--(0.2313208612697923,-0.5099717981147311)--(0.25217316739784573,-0.45569324526340765)--(0.27302547352589923,-0.4014146924120838)--(0.2938777796539527,-0.3471361395607612)--(0.3147300857820061,-0.2928575867094373)--(0.3355823919100596,-0.2385790338581148)--(0.35643469803811306,-0.18430048100679133)--(0.37728700416616656,-0.1300219281554688)--(0.39813931029422,-0.0757433753041449)--(0.41899161642227345,-0.021464822452822375)--(0.43984392255032695,0.03281373039850108)--(0.4606962286783804,0.0870922832498236)--(0.48154853480643384,0.1413708361011475)--(0.5024008409344873,0.19564938895247097)--(0.5232531470625408,0.2499279418037935)--(0.5441054531905942,0.3042064946551174)--(0.5649577593186478,0.35848504750643995)--(0.5858100654467012,0.4127636003577638)--(0.6066623715747547,0.4670421532090864)--(0.6275146777028081,0.5213207060604098)--(0.6483669838308616,0.5755992589117324)--(0.6692192899589151,0.6298778117630562)--(0.6900715960869686,0.6841563646143788)--(0.7,0.7075324268706145)--(0.7,0.7124675731293848)--(0.6900715960869686,0.735843635385621)--(0.6692192899589151,0.7901221882369436)--(0.6483669838308616,0.8444007410882675)--(0.6275146777028081,0.89867929393959)--(0.6066623715747547,0.9529578467909139)--(0.5858100654467012,1.0072363996422364)--(0.5649577593186478,1.0615149524935599)--(0.5441054531905942,1.1157935053448824)--(0.5232531470625408,1.1700720581962063)--(0.5024008409344873,1.2243506110475288)--(0.48154853480643384,1.2786291638988525)--(0.4606962286783804,1.3329077167501762)--(0.43984392255032695,1.3871862696014987)--(0.41899161642227345,1.4414648224528224)--(0.39813931029422,1.495743375304145)--(0.37728700416616656,1.5500219281554686)--(0.35643469803811306,1.6043004810067911)--(0.3355823919100596,1.658579033858115)--(0.3147300857820061,1.7128575867094373)--(0.2938777796539527,1.7671361395607612)--(0.27302547352589923,1.8214146924120838)--(0.25217316739784573,1.8756932452634074)--(0.2313208612697923,1.929971798114731)--(0.21046855514173882,1.9842503509660536)--(0.18961624901368535,2.0385289038173773)--(0.16876394288563187,2.0928074566687)--(0.1479116367575784,2.1470860095200237)--(0.12705933062952496,2.201364562371346)--(0.10620702450147149,2.2556431152226697)--(0.08535471837341801,2.3099216680739922)--(0.06450241224536456,2.3642002209253157)--(0.04365010611731155,2.4184787737766387)--(0.02279779998925762,2.4727573266279625)--(0.0019454938612041542,2.527035879479286)--(0,2.5321);
						\draw[line width=0pt,dash pattern=on 1pt off 1pt,color=zzffqq,fill=zzffqq,fill opacity=0.25](0,2.5321)--(0,-1.1121)--(0.0019454938612041542,-1.1070358794792863)--(0.02279779998925762,-1.052757326627962)--(0.04365010611731155,-0.9984787737766385)--(0.06450241224536456,-0.944200220925316)--(0.08535471837341801,-0.8899216680739925)--(0.10620702450147149,-0.83564311522267)--(0.12705933062952496,-0.781364562371346)--(0.1479116367575784,-0.7270860095200236)--(0.16876394288563187,-0.6728074566686996)--(0.18961624901368535,-0.6185289038173771)--(0.21046855514173882,-0.5642503509660536)--(0.2313208612697923,-0.5099717981147311)--(0.25217316739784573,-0.45569324526340765)--(0.27302547352589923,-0.4014146924120838)--(0.2938777796539527,-0.3471361395607612)--(0.3147300857820061,-0.2928575867094373)--(0.3355823919100596,-0.2385790338581148)--(0.35643469803811306,-0.18430048100679133)--(0.37728700416616656,-0.1300219281554688)--(0.39813931029422,-0.0757433753041449)--(0.41899161642227345,-0.021464822452822375)--(0.43984392255032695,0.03281373039850108)--(0.4606962286783804,0.0870922832498236)--(0.48154853480643384,0.1413708361011475)--(0.5024008409344873,0.19564938895247097)--(0.5232531470625408,0.2499279418037935)--(0.5441054531905942,0.3042064946551174)--(0.5649577593186478,0.35848504750643995)--(0.5858100654467012,0.4127636003577638)--(0.6066623715747547,0.4670421532090864)--(0.6275146777028081,0.5213207060604098)--(0.6483669838308616,0.5755992589117324)--(0.6692192899589151,0.6298778117630562)--(0.6900715960869686,0.6841563646143788)--(0.7,0.7075324268706145)--(0.7,0.7124675731293848)--(0.6900715960869686,0.735843635385621)--(0.6692192899589151,0.7901221882369436)--(0.6483669838308616,0.8444007410882675)--(0.6275146777028081,0.89867929393959)--(0.6066623715747547,0.9529578467909139)--(0.5858100654467012,1.0072363996422364)--(0.5649577593186478,1.0615149524935599)--(0.5441054531905942,1.1157935053448824)--(0.5232531470625408,1.1700720581962063)--(0.5024008409344873,1.2243506110475288)--(0.48154853480643384,1.2786291638988525)--(0.4606962286783804,1.3329077167501762)--(0.43984392255032695,1.3871862696014987)--(0.41899161642227345,1.4414648224528224)--(0.39813931029422,1.495743375304145)--(0.37728700416616656,1.5500219281554686)--(0.35643469803811306,1.6043004810067911)--(0.3355823919100596,1.658579033858115)--(0.3147300857820061,1.7128575867094373)--(0.2938777796539527,1.7671361395607612)--(0.27302547352589923,1.8214146924120838)--(0.25217316739784573,1.8756932452634074)--(0.2313208612697923,1.929971798114731)--(0.21046855514173882,1.9842503509660536)--(0.18961624901368535,2.0385289038173773)--(0.16876394288563187,2.0928074566687)--(0.1479116367575784,2.1470860095200237)--(0.12705933062952496,2.201364562371346)--(0.10620702450147149,2.2556431152226697)--(0.08535471837341801,2.3099216680739922)--(0.06450241224536456,2.3642002209253157)--(0.04365010611731155,2.4184787737766387)--(0.02279779998925762,2.4727573266279625)--(0.0019454938612041542,2.527035879479286)--(0,2.5321);
						\draw[line width=0pt,dash pattern=on 1pt off 1pt,color=zzffqq,fill=zzffqq,fill opacity=0.25](0,2.5321)--(0,-1.1121)--(0.0019454938612041542,-1.1070358794792863)--(0.02279779998925762,-1.052757326627962)--(0.04365010611731155,-0.9984787737766385)--(0.06450241224536456,-0.944200220925316)--(0.08535471837341801,-0.8899216680739925)--(0.10620702450147149,-0.83564311522267)--(0.12705933062952496,-0.781364562371346)--(0.1479116367575784,-0.7270860095200236)--(0.16876394288563187,-0.6728074566686996)--(0.18961624901368535,-0.6185289038173771)--(0.21046855514173882,-0.5642503509660536)--(0.2313208612697923,-0.5099717981147311)--(0.25217316739784573,-0.45569324526340765)--(0.27302547352589923,-0.4014146924120838)--(0.2938777796539527,-0.3471361395607612)--(0.3147300857820061,-0.2928575867094373)--(0.3355823919100596,-0.2385790338581148)--(0.35643469803811306,-0.18430048100679133)--(0.37728700416616656,-0.1300219281554688)--(0.39813931029422,-0.0757433753041449)--(0.41899161642227345,-0.021464822452822375)--(0.43984392255032695,0.03281373039850108)--(0.4606962286783804,0.0870922832498236)--(0.48154853480643384,0.1413708361011475)--(0.5024008409344873,0.19564938895247097)--(0.5232531470625408,0.2499279418037935)--(0.5441054531905942,0.3042064946551174)--(0.5649577593186478,0.35848504750643995)--(0.5858100654467012,0.4127636003577638)--(0.6066623715747547,0.4670421532090864)--(0.6275146777028081,0.5213207060604098)--(0.6483669838308616,0.5755992589117324)--(0.6692192899589151,0.6298778117630562)--(0.6900715960869686,0.6841563646143788)--(0.7,0.7075324268706145)--(0.7,0.7124675731293848)--(0.6900715960869686,0.735843635385621)--(0.6692192899589151,0.7901221882369436)--(0.6483669838308616,0.8444007410882675)--(0.6275146777028081,0.89867929393959)--(0.6066623715747547,0.9529578467909139)--(0.5858100654467012,1.0072363996422364)--(0.5649577593186478,1.0615149524935599)--(0.5441054531905942,1.1157935053448824)--(0.5232531470625408,1.1700720581962063)--(0.5024008409344873,1.2243506110475288)--(0.48154853480643384,1.2786291638988525)--(0.4606962286783804,1.3329077167501762)--(0.43984392255032695,1.3871862696014987)--(0.41899161642227345,1.4414648224528224)--(0.39813931029422,1.495743375304145)--(0.37728700416616656,1.5500219281554686)--(0.35643469803811306,1.6043004810067911)--(0.3355823919100596,1.658579033858115)--(0.3147300857820061,1.7128575867094373)--(0.2938777796539527,1.7671361395607612)--(0.27302547352589923,1.8214146924120838)--(0.25217316739784573,1.8756932452634074)--(0.2313208612697923,1.929971798114731)--(0.21046855514173882,1.9842503509660536)--(0.18961624901368535,2.0385289038173773)--(0.16876394288563187,2.0928074566687)--(0.1479116367575784,2.1470860095200237)--(0.12705933062952496,2.201364562371346)--(0.10620702450147149,2.2556431152226697)--(0.08535471837341801,2.3099216680739922)--(0.06450241224536456,2.3642002209253157)--(0.04365010611731155,2.4184787737766387)--(0.02279779998925762,2.4727573266279625)--(0.0019454938612041542,2.527035879479286)--(0,2.5321);
						\draw[line width=0pt,dash pattern=on 1pt off 1pt,color=zzffqq,fill=zzffqq,fill opacity=0.25](0,2.5321)--(0,-1.1121)--(0.0019454938612041542,-1.1070358794792863)--(0.02279779998925762,-1.052757326627962)--(0.04365010611731155,-0.9984787737766385)--(0.06450241224536456,-0.944200220925316)--(0.08535471837341801,-0.8899216680739925)--(0.10620702450147149,-0.83564311522267)--(0.12705933062952496,-0.781364562371346)--(0.1479116367575784,-0.7270860095200236)--(0.16876394288563187,-0.6728074566686996)--(0.18961624901368535,-0.6185289038173771)--(0.21046855514173882,-0.5642503509660536)--(0.2313208612697923,-0.5099717981147311)--(0.25217316739784573,-0.45569324526340765)--(0.27302547352589923,-0.4014146924120838)--(0.2938777796539527,-0.3471361395607612)--(0.3147300857820061,-0.2928575867094373)--(0.3355823919100596,-0.2385790338581148)--(0.35643469803811306,-0.18430048100679133)--(0.37728700416616656,-0.1300219281554688)--(0.39813931029422,-0.0757433753041449)--(0.41899161642227345,-0.021464822452822375)--(0.43984392255032695,0.03281373039850108)--(0.4606962286783804,0.0870922832498236)--(0.48154853480643384,0.1413708361011475)--(0.5024008409344873,0.19564938895247097)--(0.5232531470625408,0.2499279418037935)--(0.5441054531905942,0.3042064946551174)--(0.5649577593186478,0.35848504750643995)--(0.5858100654467012,0.4127636003577638)--(0.6066623715747547,0.4670421532090864)--(0.6275146777028081,0.5213207060604098)--(0.6483669838308616,0.5755992589117324)--(0.6692192899589151,0.6298778117630562)--(0.6900715960869686,0.6841563646143788)--(0.7,0.7075324268706145)--(0.7,0.7124675731293848)--(0.6900715960869686,0.735843635385621)--(0.6692192899589151,0.7901221882369436)--(0.6483669838308616,0.8444007410882675)--(0.6275146777028081,0.89867929393959)--(0.6066623715747547,0.9529578467909139)--(0.5858100654467012,1.0072363996422364)--(0.5649577593186478,1.0615149524935599)--(0.5441054531905942,1.1157935053448824)--(0.5232531470625408,1.1700720581962063)--(0.5024008409344873,1.2243506110475288)--(0.48154853480643384,1.2786291638988525)--(0.4606962286783804,1.3329077167501762)--(0.43984392255032695,1.3871862696014987)--(0.41899161642227345,1.4414648224528224)--(0.39813931029422,1.495743375304145)--(0.37728700416616656,1.5500219281554686)--(0.35643469803811306,1.6043004810067911)--(0.3355823919100596,1.658579033858115)--(0.3147300857820061,1.7128575867094373)--(0.2938777796539527,1.7671361395607612)--(0.27302547352589923,1.8214146924120838)--(0.25217316739784573,1.8756932452634074)--(0.2313208612697923,1.929971798114731)--(0.21046855514173882,1.9842503509660536)--(0.18961624901368535,2.0385289038173773)--(0.16876394288563187,2.0928074566687)--(0.1479116367575784,2.1470860095200237)--(0.12705933062952496,2.201364562371346)--(0.10620702450147149,2.2556431152226697)--(0.08535471837341801,2.3099216680739922)--(0.06450241224536456,2.3642002209253157)--(0.04365010611731155,2.4184787737766387)--(0.02279779998925762,2.4727573266279625)--(0.0019454938612041542,2.527035879479286)--(0,2.5321);
						\draw[line width=0pt,dash pattern=on 1pt off 1pt,color=qqqqff,fill=qqqqff,fill opacity=0.25](1.6781558199939908,3.2561395994443574)--(1.6701299841054813,3.235248348626568)--(1.6492776779774279,3.1809697957752454)--(1.6284253718493744,3.126691242923922)--(1.607573065721321,3.0724126900725994)--(1.5867207595932675,3.0181341372212755)--(1.565868453465214,2.963855584369952)--(1.5450161473371606,2.9095770315186296)--(1.5241638412091072,2.8552984786673057)--(1.5033115350810538,2.801019925815983)--(1.4824592289530003,2.7467413729646593)--(1.4616069228249466,2.692462820113337)--(1.4407546166968932,2.6381842672620133)--(1.4199023105688398,2.5839057144106907)--(1.3990500044407863,2.5296271615593673)--(1.3781976983127329,2.4753486087080443)--(1.3573453921846794,2.421070055856721)--(1.336493086056626,2.366791503005397)--(1.3156407799285725,2.3125129501540744)--(1.294788473800519,2.2582343973027506)--(1.2739361676724656,2.2039558444514284)--(1.253083861544412,2.1496772916001046)--(1.2322315554163585,2.095398738748782)--(1.211379249288305,2.0411201858974586)--(1.1905269431602516,1.9868416330461358)--(1.1696746370321982,1.9325630801948122)--(1.1488223309041448,1.8782845273434896)--(1.1279700247760913,1.824005974492166)--(1.1071177186480379,1.769727421640842)--(1.0862654125199844,1.7154488687895195)--(1.065413106391931,1.6611703159381959)--(1.0445608002638775,1.6068917630868733)--(1.0237084941358239,1.5526132102355497)--(1.0028561880077704,1.498334657384227)--(0.982003881879717,1.4440561045329035)--(0.9611515757516635,1.389777551681581)--(0.9402992696236101,1.3354989988302572)--(0.9194469634955567,1.2812204459789336)--(0.8985946573675032,1.226941893127611)--(0.8777423512394497,1.1726633402762872)--(0.8568900451113962,1.1183847874249646)--(0.8360377389833428,1.0641062345736412)--(0.8151854328552893,1.0098276817223186)--(0.7943331267272359,0.9555491288709946)--(0.7734808205991823,0.9012705760196721)--(0.7526285144711289,0.8469920231683487)--(0.7317762083430754,0.7927134703170261)--(0.710923902215022,0.7384349174657022)--(0.7057108256830086,0.7248652792528711)--(0.7031042874170019,0.7180804601464558)--(0.7004977491509953,0.711295641040041)--(0.7,0.7124675731293848)--(0.7,0.7075324268706145)--(0.7004977491509953,0.7087043589599589)--(0.7031042874170019,0.7019195398535439)--(0.7057108256830086,0.6951347207471287)--(0.710923902215022,0.6815650825342976)--(0.7317762083430754,0.6272865296829742)--(0.7526285144711289,0.5730079768316516)--(0.7734808205991823,0.5187294239803276)--(0.7943331267272359,0.4644508711290051)--(0.8151854328552893,0.41017231827768125)--(0.8360377389833428,0.3558937654263587)--(0.8568900451113962,0.30161521257503526)--(0.8777423512394497,0.24733665972371274)--(0.8985946573675032,0.19305810687238883)--(0.9194469634955567,0.1387795540210663)--(0.9402992696236101,0.08450100116974285)--(0.9611515757516635,0.030222448318418933)--(0.982003881879717,-0.02405610453290359)--(1.0028561880077704,-0.07833465738422704)--(1.0237084941358239,-0.13261321023554956)--(1.0445608002638775,-0.18689176308687347)--(1.065413106391931,-0.24117031593819602)--(1.0862654125199844,-0.2954488687895195)--(1.1071177186480379,-0.349727421640842)--(1.1279700247760913,-0.4040059744921659)--(1.1488223309041448,-0.4582845273434898)--(1.1696746370321982,-0.5125630801948123)--(1.1905269431602516,-0.5668416330461358)--(1.211379249288305,-0.6211201858974583)--(1.2322315554163585,-0.6753987387487822)--(1.253083861544412,-0.7296772916001047)--(1.2739361676724656,-0.7839558444514282)--(1.294788473800519,-0.8382343973027507)--(1.3156407799285725,-0.8925129501540746)--(1.336493086056626,-0.9467915030053967)--(1.3573453921846794,-1.0010700558567212)--(1.3781976983127329,-1.0553486087080446)--(1.3990500044407863,-1.1096271615593671)--(1.4199023105688398,-1.1639057144106906)--(1.4407546166968932,-1.2181842672620131)--(1.4616069228249466,-1.2724628201133366)--(1.4824592289530003,-1.3267413729646589)--(1.5033115350810538,-1.3810199258159834)--(1.5241638412091072,-1.435298478667306)--(1.5450161473371606,-1.4895770315186294)--(1.565868453465214,-1.543855584369952)--(1.5867207595932675,-1.5981341372212754)--(1.607573065721321,-1.6524126900725997)--(1.6284253718493744,-1.7066912429239223)--(1.6492776779774279,-1.7609697957752457)--(1.6701299841054813,-1.8152483486265683)--(1.690982290233535,-1.8695269014778917)--(1.7118345963615884,-1.9238054543292142)--(1.7326869024896419,-1.9780840071805377)--(1.7535392086176953,-2.03236256003186)--(1.7743915147457487,-2.0866411128831843)--(1.7952438208738022,-2.140919665734507)--(1.8160961270018552,-2.1951982185858294)--(1.836948433129909,-2.249476771437155)--(1.8578007392579625,-2.303755324288477)--(1.878653045386016,-2.3580338771397997)--(1.8882972439927612,-2.383137726113157)--(3.06926886601463,-2.383137726113157)--(3.06926886601463,3.2561395994443574);
						\draw[line width=0pt,dash pattern=on 1pt off 1pt,color=qqqqff,fill=qqqqff,fill opacity=0.25](1.6781558199939908,3.2561395994443574)--(1.6701299841054813,3.235248348626568)--(1.6492776779774279,3.1809697957752454)--(1.6284253718493744,3.126691242923922)--(1.607573065721321,3.0724126900725994)--(1.5867207595932675,3.0181341372212755)--(1.565868453465214,2.963855584369952)--(1.5450161473371606,2.9095770315186296)--(1.5241638412091072,2.8552984786673057)--(1.5033115350810538,2.801019925815983)--(1.4824592289530003,2.7467413729646593)--(1.4616069228249466,2.692462820113337)--(1.4407546166968932,2.6381842672620133)--(1.4199023105688398,2.5839057144106907)--(1.3990500044407863,2.5296271615593673)--(1.3781976983127329,2.4753486087080443)--(1.3573453921846794,2.421070055856721)--(1.336493086056626,2.366791503005397)--(1.3156407799285725,2.3125129501540744)--(1.294788473800519,2.2582343973027506)--(1.2739361676724656,2.2039558444514284)--(1.253083861544412,2.1496772916001046)--(1.2322315554163585,2.095398738748782)--(1.211379249288305,2.0411201858974586)--(1.1905269431602516,1.9868416330461358)--(1.1696746370321982,1.9325630801948122)--(1.1488223309041448,1.8782845273434896)--(1.1279700247760913,1.824005974492166)--(1.1071177186480379,1.769727421640842)--(1.0862654125199844,1.7154488687895195)--(1.065413106391931,1.6611703159381959)--(1.0445608002638775,1.6068917630868733)--(1.0237084941358239,1.5526132102355497)--(1.0028561880077704,1.498334657384227)--(0.982003881879717,1.4440561045329035)--(0.9611515757516635,1.389777551681581)--(0.9402992696236101,1.3354989988302572)--(0.9194469634955567,1.2812204459789336)--(0.8985946573675032,1.226941893127611)--(0.8777423512394497,1.1726633402762872)--(0.8568900451113962,1.1183847874249646)--(0.8360377389833428,1.0641062345736412)--(0.8151854328552893,1.0098276817223186)--(0.7943331267272359,0.9555491288709946)--(0.7734808205991823,0.9012705760196721)--(0.7526285144711289,0.8469920231683487)--(0.7317762083430754,0.7927134703170261)--(0.710923902215022,0.7384349174657022)--(0.7057108256830086,0.7248652792528711)--(0.7031042874170019,0.7180804601464558)--(0.7004977491509953,0.711295641040041)--(0.7,0.7124675731293848)--(0.7,0.7075324268706145)--(0.7004977491509953,0.7087043589599589)--(0.7031042874170019,0.7019195398535439)--(0.7057108256830086,0.6951347207471287)--(0.710923902215022,0.6815650825342976)--(0.7317762083430754,0.6272865296829742)--(0.7526285144711289,0.5730079768316516)--(0.7734808205991823,0.5187294239803276)--(0.7943331267272359,0.4644508711290051)--(0.8151854328552893,0.41017231827768125)--(0.8360377389833428,0.3558937654263587)--(0.8568900451113962,0.30161521257503526)--(0.8777423512394497,0.24733665972371274)--(0.8985946573675032,0.19305810687238883)--(0.9194469634955567,0.1387795540210663)--(0.9402992696236101,0.08450100116974285)--(0.9611515757516635,0.030222448318418933)--(0.982003881879717,-0.02405610453290359)--(1.0028561880077704,-0.07833465738422704)--(1.0237084941358239,-0.13261321023554956)--(1.0445608002638775,-0.18689176308687347)--(1.065413106391931,-0.24117031593819602)--(1.0862654125199844,-0.2954488687895195)--(1.1071177186480379,-0.349727421640842)--(1.1279700247760913,-0.4040059744921659)--(1.1488223309041448,-0.4582845273434898)--(1.1696746370321982,-0.5125630801948123)--(1.1905269431602516,-0.5668416330461358)--(1.211379249288305,-0.6211201858974583)--(1.2322315554163585,-0.6753987387487822)--(1.253083861544412,-0.7296772916001047)--(1.2739361676724656,-0.7839558444514282)--(1.294788473800519,-0.8382343973027507)--(1.3156407799285725,-0.8925129501540746)--(1.336493086056626,-0.9467915030053967)--(1.3573453921846794,-1.0010700558567212)--(1.3781976983127329,-1.0553486087080446)--(1.3990500044407863,-1.1096271615593671)--(1.4199023105688398,-1.1639057144106906)--(1.4407546166968932,-1.2181842672620131)--(1.4616069228249466,-1.2724628201133366)--(1.4824592289530003,-1.3267413729646589)--(1.5033115350810538,-1.3810199258159834)--(1.5241638412091072,-1.435298478667306)--(1.5450161473371606,-1.4895770315186294)--(1.565868453465214,-1.543855584369952)--(1.5867207595932675,-1.5981341372212754)--(1.607573065721321,-1.6524126900725997)--(1.6284253718493744,-1.7066912429239223)--(1.6492776779774279,-1.7609697957752457)--(1.6701299841054813,-1.8152483486265683)--(1.690982290233535,-1.8695269014778917)--(1.7118345963615884,-1.9238054543292142)--(1.7326869024896419,-1.9780840071805377)--(1.7535392086176953,-2.03236256003186)--(1.7743915147457487,-2.0866411128831843)--(1.7952438208738022,-2.140919665734507)--(1.8160961270018552,-2.1951982185858294)--(1.836948433129909,-2.249476771437155)--(1.8578007392579625,-2.303755324288477)--(1.878653045386016,-2.3580338771397997)--(1.8882972439927612,-2.383137726113157)--(3.06926886601463,-2.383137726113157)--(3.06926886601463,3.2561395994443574);
						\draw (0.7,0.71) node {\LARGE $\cdot$};
						\draw (1.0,0.71) node {\scriptsize $z_0$};
						\draw (0.7,-0.718) node {\LARGE $\cdot$};
						\draw (1.0,-0.71) node {\scriptsize $\overline{z_0}$};
						\draw[color=blue] (1.5,0.9) node {$R_4$};
						\draw[color=red] (0.85,1.8) node {$R_1$};
						\draw[color=green] (0.39,0.71) node {$R_3$};
						\draw[color=yellow] (0.92,-1.5) node {$R_2$};
					\end{axis}
				\end{tikzpicture}
				\caption{The regions $R_1,R_2,R_3,R_4$ as described in the proof above, with $c_0 = 0.7.$ Notice that we dropped the condition that $\text{dist}(s,\supp(\tilde{\nu}))$ is small for a clearer visualisation.}
			\end{figure}
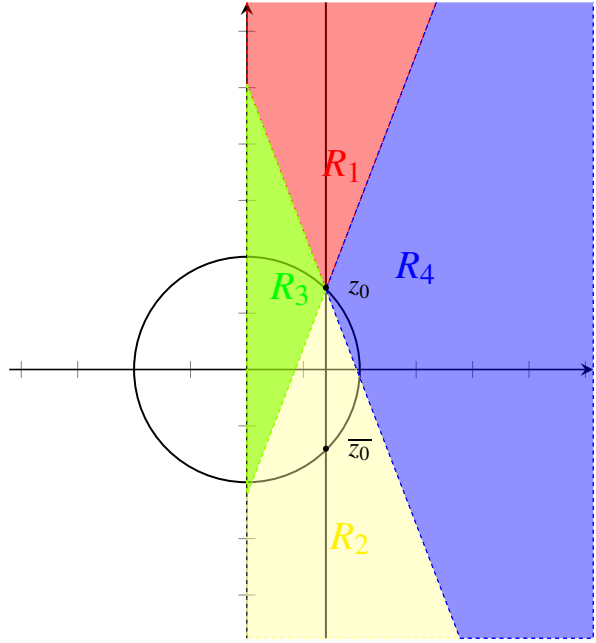
			
			\item $|sF_{\tilde{\nu}}(s)|$ is bounded as $s \to \infty,$ which follows from \eqref{eq:Cauchy-transform}.  
		\end{enumerate}
		
		From these properties, we are led to consider the function $H(s) = (s-z_0)(s-\overline{z_0})F_{\tilde{\nu}}(s).$ By the considerations above, $H$ is an entire function, bounded by a polynomial of degree 1. Thus, $H$ is itself a polynomial of degree 1, say, $H(s) = \alpha s + \beta.$ Then 
		$$
		F_{\tilde{\nu}}(s) = \frac{\alpha s + \beta}{(s-z_0)(s-\overline{z_0})} = \frac{\gamma}{z_0-s} + \frac{\sigma}{\overline{z_0}-s},
		$$
		for some $\gamma,\sigma \in \C.$ In order to finish, we employ the a theorem of  F. Riesz and M. Riesz \cite{Garnett} saying that for two measures $\mu_1$ and $\mu_2$ on the unit circle we have
		\begin{equation}\label{eq:unique-Cauchy-transform} 
			\int_{\mathbb{S}^1} \frac{1}{z-s} \, d(\mu_1 - \mu_2) (z) = 0\, \, \forall \, \, s \in \mathbb{D} \iff d\mu_1 - d\mu_2 = \phi(s) dm(s), 
		\end{equation}
		where $\phi \in H^1(\partial \mathbb{D}),$ and $dm$ denotes the arclength measure on the circle $\mathbb{S}^1.$  Applied to our case, we obtain
		$$
		\tilde{\nu} = \gamma \cdot \delta_{z_0} + \sigma \cdot \delta_{\overline{z_0}} + \phi(s) dm(s),
		$$
		for some $\phi \in H^1(\partial\mathbb{D}).$ But $\supp(\tilde{\nu} - \gamma\cdot \delta_{z_0} - \sigma \cdot \delta_{\overline{z_0}}) \subset \mathbb{S}^1 \cap \{\text{Re}(z) \le c_0\}.$ This shows that $\phi$ vanishes on the arc joining $z_0$ and $\overline{z_0}.$ Classical uniqueness result for functions from the Hardy space implies that $\phi \equiv 0.$ This implies that  
		$$ \supp(\nu) \subset \mathbb{S}^1 \cap \{ \text{Re}(z) \ge c_0\},$$
		which finishes the proof of Lemma \ref{lemma-support-measure-circle}. 
	\end{proof}
	
	With Lemma \ref{lemma-support-measure-circle} at our disposal, we may proceed to the proof of Theorem \ref{thm:second-laplace-full}. 
	
	\begin{proof}[Proof of Theorem \ref{thm:second-laplace-full}]  We first rewrite what we wish to prove in terms of Laplace transforms of measures on the circle, as done above. 
		
		For a finite measure $\mu$  on the interval $[0,1],$ let $\pi_1:\mathbb{S}^1 \to [-1,1]$ denote the projection onto the first coordinate, and let $p:[0,1] \to A = \{ z = e^{i \theta}, \, \theta \in [0,\pi/2]\}$ denote the inverse of this map restricted to the set $A.$ We then let 
		$$ \nu = p_{\ast}(\mu) $$
		be the pushforward measure of $\mu$ to the circle through $p.$ We readily see that this measure is finite, and $\supp(\nu) \subset A \subset \mathbb{S}^1 \cap \{\text{Re}(s) \ge 0\}.$ Finally, from the definition of $\nu$ and $\varphi,$ we may write
		$$ \varphi(x) = \int_{\mathbb{S}^1} e^{-z\pi |x|^2} \, d \nu(z) = \mathcal{L}\nu(\pi|x|^2).$$
		
		\noindent\textit{Proof of Part (1).} As $\varphi \in E^2_a,$ the main result shows that $\varphi \in E^{\infty}_{a-\eps}, \, \forall \, \eps > 0.$ But, by the correspondence above, we have $|\mathcal{L}\nu(t)| \le C e^{-(a-\eps)t},$ for any $\eps > 0.$ By Lemma \ref{lemma-support-measure-circle}, $\supp(\nu) \subset \mathbb{S}^1 \cap \{\text{Re}(s) \ge a - \eps\}, \, \forall \,\, \eps > 0.$ Thus, $\supp(\nu) \subset \mathbb{S}^1 \cap \{\text{Re}(s) \ge a \}.$ But this is equivalent to $\supp(\mu) \subset [a,1].$ This plainly implies that $|\varphi(x)|e^{a \pi |x|^2} \in L^{\infty}(\R).$ 
		
		On the other hand, we see that $\widehat{\varphi}$ may be written as 
		$$ \widehat{\varphi}(\xi) = \int_0^1 \overline{\mathcal{G}_r(\xi)} \cdot \frac{1}{(r + i \sqrt{1-r^2})^{1/2}} \, d \mu(r) =: \int_0^1  \overline{\mathcal{G}_r(\xi)} \, d \tilde{\mu}(r).$$ 
		The measure $\tilde{\mu}$ is again a \emph{finite} measure, and one can repeat the argument above, now using the ``conjugate'' map $\overline{p} : [0,1] \to \overline{A} = \{ z = e^{i\theta}, \, \theta \in [-\pi/2,0]\}$ to define the pushforward measure. This directly implies that $|\widehat{\varphi}(\xi)|e^{a \pi |\xi|^2} \in L^{\infty}(\R),$ which finally shows that $\varphi \in E^{\infty}_a,$ as desired. \\
		
		\noindent\textit{Proof of Part (2).} We first notice that for $\beta \not\in \{-\pi/4 - k\pi/2, k \in \Z\},$ Theorem \ref{thm:vemuri} covers this part. Thus, we may suppose without loss of generality that $\beta = -\frac{\pi}{4} - \frac{k \pi}{2}, \, k \in \Z.$ 
		
		If $\varphi \in E^{\infty}_a,$ then either Theorem \ref{thm:vemuri} -- or even Vemuri's Theorem  3.1 -- shows that $\mathcal{F}_{-\pi/4 - k \pi/2} \varphi \in E^{\infty}_{\tanh(\alpha) - \eps}, \, \forall \,\, \eps> 0, \, \forall \,\, k \in \Z.$ Now, the computation on \eqref{eq:gaussian-fractional} with $\lambda = r + i\sqrt{1-r^2}, \, \beta = - \pi/4 + k\pi/2$ shows that   
		\begin{equation}\label{eq:frft-gauss}
			|\mathcal{F}_{-\pi/4 - k \pi/2} \mathcal{G}_r(x)| \lesssim \exp\left( - \frac{r \pi |x|^2}{1 -(-1)^k \sqrt{1-r^2}} \right). 
		\end{equation}
		On the other hand, as $\varphi \in E^{\infty}_{\tanh(2\alpha)},$ using the same argument as in the proof of Part (1) above, we have that $\supp(\mu) \subset [\tanh(2\alpha),1].$ Thus, as we know that the functions of $r \in [0,1]$ 
		$$
		r\mapsto \frac{r}{1 - \sqrt{1-r^2}}, \,\,\,\, r \mapsto \frac{r}{1+\sqrt{1-r^2}}
		$$
		are, respectively, decreasing and increasing, \eqref{eq:frft-gauss} implies that 
		\begin{align*}
			|\mathcal{F}_{-\pi/4 - k \pi/2} \varphi(x)| & \lesssim \|\mu\|_{TV} \max \left\{ \exp\left( - \frac{a \pi |x|^2}{1 + \sqrt{1-a^2}} \right), \exp(-\pi |x|^2)  \right\} \cr 
			& = \|\mu\|_{TV}  \max \left\{ \exp(-\tanh(\alpha) \pi |x|^2), \exp(-\pi |x|^2)\right\}, \cr 
		\end{align*}
		as $a = \tanh(2\alpha).$ This concludes the proof of Part (2), and thus also that of Theorem \ref{thm:second-laplace-full}. 
	\end{proof}
	
	\begin{remark}
		Using the notation employed in the statement of Conjecture \ref{conjecture:upgraded-vemuri}, we have, for given $t \in \R$ and $\varepsilon>0,$ that $\mathcal{F}_{-4\pi t}f \in E_{\Omega_{\alpha}(t) - \varepsilon}^{\infty}$ for any $t\in\R.$ A calculation using \eqref{eq:gaussian-fractional} and the same strategy as in Part (2) of the proof of Theorem \ref{thm:second-laplace-full} above shows that, in fact, $\Phi(\cdot,t) \in E_{\Omega_{\alpha}(t)}^{\infty}.$ This shows the validity of Conjecture \ref{conjecture:upgraded-vemuri} for the class of Laplace transforms with support on the circle discussed above. 
	\end{remark}
	
	\section*{Acknowledgements} 
	
	We would like to express our gratitude towards Prof. Dimitar K. Dimitrov, who made several comments and remarks which helped shape the manuscript, to Prof. Sundaram Thangavelu, who pointed to us some results related to the work of Vemuri which led to its higher dimensional case, and to Biagio Cassano, for directing our attention to closely related work. We would also like to thank Danylo Radchenko and Prof. Christoph M. Thiele for valuable discussions regarding the applications to different contexts and the endpoint result. 
	
	A.K. was supported by Grant 275113 of the Research Council of Norway. J.P.G.R. acknowledges financial support by the European Research Council under the Grant Agreement No. 721675 “Regularity and Stability in Partial Differential Equations (RSPDE)''. 
	\nocite{Farah-Mokni}
	\addcontentsline{toc}{section}{References}
	\bibliographystyle{alpha}
	\bibliography{biblio}
	
\end{document}